\documentclass[11pt,a4paper]{article}

\usepackage[utf8]{inputenc}
\usepackage{amsmath,amsfonts,amssymb,amsthm,accents,dsfont,bbm}
\usepackage{graphicx,calc}
\usepackage{setspace,tikz,mathrsfs}
\usepackage[top=3cm,right=3cm,left=3cm,bottom=3cm]{geometry}
\usepackage[shortlabels]{enumitem}
\usepackage{chngcntr}
\usepackage[nottoc]{tocbibind}
\usepackage{url}
\usepackage[colorlinks,bookmarksopen,bookmarksnumbered,citecolor=red,urlcolor=red]{hyperref}

\author{Kieran Jarrett}

\newtheoremstyle{standard}
  {11pt} 
  {} 
  {\itshape} 
  {} 
  {\bfseries} 
  {} 
  {.5em} 
  {} 
  
\theoremstyle{standard}
\newtheorem{theorem}{Theorem}[section]
\newtheorem{lem}[theorem]{Lemma}
\newtheorem{cor}[theorem]{Corollary}
\newtheorem{defn}[theorem]{Definition}
\newtheorem{prop}[theorem]{Proposition}

\newtheorem*{theorem*}{Theorem}

\newtheoremstyle{rem}
  {9pt} 
  {} 
  {} 
  {} 
  {\bfseries} 
  {} 
  {.5em} 
  {} 
\theoremstyle{rem}

\newtheoremstyle{rems}
  {9pt} 
  {} 
  {} 
  {} 
  {\bfseries} 
  {} 
  {.5em} 
  {} 

\theoremstyle{rems}

\newtheoremstyle{e.g.}
  {9pt} 
  {} 
  {} 
  {} 
  {\itshape} 
  {} 
  {.5em} 
  {} 
\theoremstyle{e.g.}
\newtheorem*{e.g.}{E.g.}

\onehalfspace

\SetLabelAlign{Center}{\hfil#1\hfil}

\newcommand{\ind}{\normalfont \textbf{1}}

\newcommand{\rmin}{\textup{rmin\,}}
\newcommand{\rmax}{\textup{rmax\,}}

\newcommand{\tb}[3]{\partial_{#1} B_{#2}(#3)}
\newcommand{\re}[1]{\textup{Re\,}{#1}}
\newcommand{\im}[1]{\textup{Im\,}{#1}}

\numberwithin{equation}{section}

\begin{document}

\title{An ergodic theorem for non-singular actions of the Heisenberg groups}
\date{}
\maketitle

\begin{abstract}
We show that there is a sequence of subsets of each discrete Heisenberg group for which the non-singular ergodic theorem holds. The sequence depends only on the group; it works for any of its non-singular actions. To do this we use a metric which was recently shown by Le Donne and Rigot to have the Besicovitch covering property and then apply an adaptation of Hochman's proof of the multiparameter non-singular ergodic theorem. An exposition of how one proves non-singular ergodic theorems of this type is also included, along with a new proof for one of the key steps.
\end{abstract}

\section{Introduction}

One of the fundamental problems in ergodic theory is to establish when an average of a function over an orbit in a dynamical system - the time average - agrees with its average over the whole space. Birkhoff's pointwise ergodic theorem states that this is the case when the dynamics are described by repeated applications of a transformation to a finite measure space, so long as the mass of each set is preserved. This theorem serves as a foundation for the class of pointwise ergodic theorems, each of which modifies or generalises this Birkhoff's result. The main result of this paper sits between two such generalisations. 

The first extends the notion of time. In Birkhoff's theorem time is discrete,  with a fixed map describing where each point in the space will be after each unit time step. In particular, if this map is invertible it induces an action of the integers on the underlying space. One can change the notion of time by considering actions of groups other than the integers, such as the reals inducing continous time. In this paper we start by considering the action of a countable group $G$ which will eventually be taken to be the discrete Heisenberg group.

The second generalisation weakens the assumption that the dynamics preserve the mass of all sets. Instead we only assume that the action is non-singular, i.e. only that mass of the null sets (those without mass) are preserved. This means that sets with positive mass can have their mass changed by the dynamics, but will not lose their mass entirely. The cost for this weakening is that time average must be weighted differently, in particular by Radon-Nikodým derivatives.

More precisely, we take $G$ to be a countable group with a non-singular and ergodic action on a standard probability space $(X,\mathcal{B},\mu)$. Each $g \in G$ then induces a non-singular map on $X$ which we also denote by $g$. The measures $\mu$ and $\mu \circ g$ are equivalent and so the Radon-Nikodým derivatives
$$\omega_g = \frac{d \mu \circ g}{d \mu}$$
are well defined and strictly positive almost everywhere. For $f \in L^1$ and $g \in G$ we let $\hat{g}f(x) = f(gx)\omega_g(x)$ for each $x \in X$.

Given such an action and a sequence ${e \in B_1 \subseteq B_2 \subseteq ...}$ of finite subsets of $G$, which we refer to as a \textit{summing sequence}, we say that \textit{the (pointwise) ergodic theorem is satisfied for the sequence $(B_k)$} if for every integrable function $f$, we have
\begin{align*}
\lim_{k \rightarrow \infty} \frac{\sum_{g \in B_k} \hat{g}f}{\sum_{g \in B_k} \hat{g}1} = \int f \, d\mu
\end{align*}
almost everywhere. 

For example, in the case of Birkhoff's theorem $B_k = \lbrace 1,...,k \rbrace$ and all the Radon-Nikodým derivatives are identically $1$, reducing the average on the left hand side to $\frac{1}{k} \sum_{i=1}^k f(ix)$ - the unweighted average over the first $k$ points in the forward orbit of $x$.

The main theorem of this paper is the following.

\begin{theorem} \label{HeisET}
There is a metric $d$ on the discrete Heisenberg group $\mathrm{H}^n = H_n(\mathbb{Z})$ for which the ergodic theorem is satisfied with $B_k = \lbrace p \in \mathrm{H}^n : d(p,0) \leq k \rbrace$.
\end{theorem}

See section \ref{Heissetup} for the precise definition of the metric $d$. We expect that similar techniques can be used to show that the corresponding result holds for the continuous Heisenberg group $\mathbb{H}^n$.

\subsubsection*{Background}

The ergodic theorem was extended separately in both of the contexts mentioned earlier. In 1944 Hurewicz proved the non-singular version of Birkhoff's theorem \cite{Hur1}, that with $G = \mathbb{Z}$ and $B_k = \lbrace 1,...,k \rbrace$, is satisfied so long as the action is conservative. The case of measure preserving actions of amenable groups was resolved far more recently by Lindenstrauss \cite{Lind1} who proved that the theorem is satisfied whenever $(B_k)$ is a tempered F{\o}lner sequence. A short inductive proof, given in \cite{Lind1}, shows that every F{\o}lner sequence has a tempered subsequence. It therefore follows that every measure preserving action of an amenable group has a summing sequence for which the ergodic theorem holds. 

Knowing these results, it is natural to ask whether given a non-singular and ergodic action of an amenable group there exists a summing sequence for which the ergodic theorem holds. In contrast with the case of integer actions, there is not a direct analogue to Lindenstrauss's result for non-singular actions; we cite work providing a counterexample below. However, there have been extensions to actions of $\mathbb{Z}^n$. 

In his paper \cite{Fel1} Feldman used an elegant method to show that the ergodic theorem holds when the summing sets $B_k$ are taken to be the balls $\lbrace u \in \mathbb{Z}^n : \|u \| \leq k \rbrace$, where $\| \cdot \|$ is the sup-norm, so long as the standard generators $e_1,...,e_n$ of $\mathbb{Z}^n$ act conservatively on $X$. Shortly after, Hochman \cite{Hoch1} used a different approach to remove this additional assumption and allow $\| \cdot \|$ to be any norm on $\mathbb{R}^n$ - under the relatively light assumption that the action is free. In recent work  with Anthony Dooley \cite{DoolJar1}, the author showed that the balls of norms can be replaced with rectangles which are symmetric around the origin with side lengths growing at arbitrarily quick and distinct rates.

There are also examples for which the ergodic theorem fails. In \cite{Hoch2}, Hochman shows that for $G = \mathbb{Z}^\infty = \bigoplus_{n=1}^\infty \mathbb{Z}$ and any choice of summing sequence there is an infinite measure  preserving action of $\mathbb{Z}^\infty$ for which the \textit{ratio} ergodic theorem fails (in fact, the ratio is shown to diverge). By transferring to an equivalent probability measure, and recalling that under our assumptions the ratio ergodic theorem is a consequence of the ergodic theorem, it follows that the ergodic theorem fails for some action of $\mathbb{Z}^\infty$ regardless of the choice of summing sequence. In the same paper it is also shown that if $G$ is taken to be the discrete Heisenberg group and $B_k = B^k$, where $B$ is the collection of standard generators of $G$, then the ratio ergodic theorem fails in a similar way for every subsequence of $(B_k)$.

The key obstacle to the validity of the ergodic theorem cited in \cite{Hoch2} is the failure of the sequence $(B_k)$ to satisfy the Besicovitch covering property (BCP), see definition \ref{BCP}. In the case of the Heisenberg group this property does not just fail for the sequence $(B^k)$, as described above, it also fails for the integer balls of the Kor\'{a}nyi distance and the Carnot-Carath\'{e}odory metric (see \cite{Rig1,SawWhee1}). These are two of the natural distances on the Heisenberg group.

However, in soon to be published work, Le Donne and Rigot \cite{LeDRig1} have show that there is a class of homogeneous metrics on the Heisenberg group which do satisfy the BCP. We recall the definition of one of these metrics in section \ref{Heissetup} and will take it to be the metric $d$ in theorem \ref{HeisET}. Though the same techniques should work for any metric in the class identified in \cite{LeDRig1} we treat just one for notational simplicity. Knowing that $d$ satisfies the BCP allows us to employ a framework to prove Theorem \ref{HeisET} developed from the one employed by Hochman in his study of $\mathbb{Z}^n$. 

The question underlying this and similar work is: given $G$ can we find a summing sequence $(B_k)$ for which the non-singular ergodic theorem holds regardless of the action? Hochman's work with $\mathbb{Z}^\infty$ means that the generality achieved by Lindenstrauss cannot replicated for non-singular actions, even in the context of abelian groups. The class of groups with for which the answer is yes is strictly smaller. On the other hand, it includes $\mathbb{Z}^n$ and $\mathrm{H}^n$ both of which are `finite dimensional' unlike $\mathbb{Z}^\infty$. These groups have various useful geometrical properties (see section 2), many of which are inherited from either $\mathbb{R}^n$ or the continous Heisenberg group $\mathbb{H}^n$. As such, it would be interesting to know whether a similar result may hold with $G = \mathbb{Q}^n$ or, more generally, countable abelian groups of finite rank. 

\hfill

\textit{Acknowledgements.} This work will be submitted as part of the author's PhD thesis and was supported by the EPSRC, UK. The author thanks his supervisor Tony Dooley for our useful discussions and the University of Technology, Sydney, for their hospitality whilst much of the work was conducted. 

\subsubsection*{Paper layout}

In section 2 we identify a small collection of geometrical properties on a group metric space which suffice to prove a non-singular ergodic theorem when $(B_k)$ is its sequence of integer balls, and then give a relatively self-contained exposition of how one does this. This includes a new proof of the main technical result from \cite{Hoch1}.

In section 3 we show that the discrete Heisenberg group satisfies each of these properties, and hence deduce Theorem \ref{HeisET}. Most of the work in this section will go into showing that the Heisenberg group equipped with the specified metric satisfies a condition based on the finite coarse dimension property used in \cite{Hoch1}.

\section{A method for proving non-singular ergodic theorems}

In this section we give a fairly self-contained exposition of a method for proving non-singular ergodic theorems when $G$ can equipped with right invariant metric $d$ with ``good" geometrical properties. This method follows the standard approach; we first show that there is a dense subset of $L^1$ on which the ergodic theorem holds and then extend to the whole of $L^1$ by using a maximal inequality. The geometry plays a central role both in the proof of the existence of such a dense set and of the maximal inequality. The summing sequence used will be the sequence of closed integer balls about the identity, i.e. the sets $B_k = B_k(e) = \lbrace g \in G : d(g,e) \leq k \rbrace$ for $k \in \mathbb{N}$. We require each of these sets to be finite.

We draw on the ideas of \cite{Aar1} and \cite{Fel1} which give elegant expositions of aspects of the method. The approach to the part specialised to this setting is based on that in \cite{Hoch1}, and in particular we provide a new proof of the main technical result in that paper.

\subsection{Construction of the dense subset}

The candidate for the dense subset of $L^1$ for which the ergodic theorem holds is
$$S = \lbrace c + h - \hat{\sigma}h : c \in \mathbb{R}, \sigma \in G, h \in L^\infty \rbrace.$$
To see that $S$ is dense in $L^1$ first recall that $L^\infty(\mu)$ can be identified with the dual of $L^1(\mu)$ and observe that $S$ is a linear subspace of $L^1(\mu)$. If $S$ is not dense we can apply the Hahn-Banach theorem to construct an $f \in L^\infty(\mu) \setminus \lbrace 0 \rbrace$ for which $\int sf \, d\mu = 0$ for all $s \in S$. It follows that it is enough to show that every $f \in L^\infty(\mu)$ such that $\int sf \, d\mu = 0$ for all $s \in S$ satisfies $f = 0$ a.e.\ . With this in mind, let us fix such an $f$. By considering the case $c = 0$ we have 
$$\int h f \, d\mu = \int (\hat{\sigma}h) f \, d\mu = \int h(x) f(\sigma^{-1}x) \, d\mu(x)$$  
for all $h \in L^\infty(\mu)$ and $\sigma \in G$. Since $L^\infty(\mu)$ is dense in $L^1$ it follows that $f \circ \sigma = f$ a.e.\ for all $\sigma \in G$, so by the ergodicity of the action $f$ is constant a.e.\ . It is easy to see this constant must be $0$, as required.

Since the integral of any $c + h - \hat{\sigma}h \in S$ is just $c$ and
$$\frac{\sum_{g \in B_k} \hat{g}(c+h - \hat{\sigma}h)}{\sum_{g \in B_k} \hat{g}1} 
= c + \frac{\sum_{g \in B_k \setminus \sigma B_k} \hat{g}h - \sum_{g \in \sigma B_k\setminus  B_k} \hat{g}h}{\sum_{g \in B_k} \hat{g}1}$$ 
as each $h \in L^\infty$ it is sufficient to prove that for all $\sigma \in G$
\begin{align*}
\frac{\sum_{g \in B_k \triangle \sigma B_k} \omega_g}{\sum_{g \in B_k} \omega_g} \rightarrow 0 \qquad \text{a.s.,} \tag{nsFC} \label{nsFC}
\end{align*}
which we will call the \textit{non-singular F{\o}lner condition} (nsFC). The proof of (nsFC) will be the focus for the rest of this section, and the approach we will use is developed from the one in \cite{Hoch1}.

In fact, this will follow largely from the metric structure we impose on the space so let us take $(M,d)$ to be a metric space and make some definitions. 

For $r > 0$ and $x \in M$ we let $B_r(x) = \lbrace y \in X : d(x,y) \leq r \rbrace$, the closed ball of radius $r$ about $x$, and we assume each such ball carries the information about its centre and radius along with it. Let $\mathcal{V}$ be a collection of balls in $M$, we let $\rmax{\mathcal{V}}$ and $\rmin{\mathcal{V}}$ denote the maximum and minimum radii, respectively, of the balls in $\mathcal{V} $. We say the collection $\mathcal{V}$ is \textit{well-separated} if the distance between each of the balls in $\mathcal{V}$ is at least $\rmin{\mathcal{V}}$. Given a finite set $E \subset X$ a \textit{carpet} over $E$ is a collection of balls $\mathcal{U} = \lbrace B_{r(x)}(x) : x \in E \rbrace$ centred in $E$. A \textit{stack} (of height $p$) over $E$ is a sequence of carpets $\mathcal{U}_1,...,\mathcal{U}_p$ over $E$. The first geometrical property we will require of the group metric space is the following.

\begin{defn}
$(M,d)$ is \emph{well-separable} if there exists $\chi \in \mathbb{N}$ such that for every finite set $E \subset M$ and every carpet $\mathcal{U}$ over $E$ there is a subcollection $\mathcal{V}$ of $\mathcal{U}$ which covers $E$ and which can be partitioned into $\chi$ well-separated subcollections.
\end{defn}

The motivation for this definition is that we immediately get the following statement, essentially by the pigeonhole principle.

\begin{lem} 
Suppose that $(M,d)$ is well-separable with constant $\chi$ and assume there is a given finite measure $\nu$ supported in a set $E$. Then given any carpet $\, \mathcal{U}$ there is well-separated subset of $\, \mathcal{U}$ which covers a set of mass $\geq (1/\chi) \nu(E)$.
\end{lem}

Now we define a concept of boundary which is in general distinct from the one given in \cite{Hoch1}, but which coincide in the context of that paper. Essentially the same definition is also used in \cite{DoolJar1}. Let $(\tilde{M},\tilde{d})$ be a metric space such that $M \subseteq \tilde{M}$ and $d = \tilde{d}|_M$, and we say that $(\tilde{M},\tilde{d})$ \textit{extends} $(M,d)$ and that $(M,d)$ is a \textit{restriction} of $(\tilde{M},\tilde{d})$ when this occurs. In this situation we then define the \textit{$t$-boundary} $\partial_t B_r(x)$ (with respect to $\tilde{M}$), where $t \geq 0$, of a ball $B_r(x)$ in $M$ by 
$$\partial_t B_r(x) = \lbrace y \in M : \tilde{d}(y,\partial \tilde{B}_r(x)) \leq t \rbrace.$$
where $\tilde{B}_r(x) = \lbrace y \in \tilde{M} : \tilde{d}(x,y) \leq r \rbrace$ and $\partial \tilde{B}_r(x) = \lbrace y \in \tilde{M} : \tilde{d}(x,y) = r \rbrace$. We also assume that spheres $\partial \tilde{B}_r(x)$ retain the information about their centre and radius. The intention is that $\tilde{M}$ will be to $M$ what $\mathbb{R}^n$ is to $\mathbb{Z}^n$.

For ease, when $(\tilde{M},\tilde{d})$ extends $(M,d)$ and we refer to points or sets being a given distance apart, this means in the metric $\tilde{d}$. Clearly this is a lower bound for the $d$-distance apart.

Given a collection $\mathcal{V}$ of balls in $M$ with extension $\tilde{M}$ we let ${\partial \mathcal{V} = \lbrace \partial \tilde{B}_r(x) : B_r(x) \in \mathcal{V} \rbrace}$, a collection of spheres in $\tilde{M}$. We also call the collection $\partial \mathcal{V}$ \textit{well-separated} if the distance between each of the spheres in $\partial \mathcal{V}$ is at least $\rmin{\mathcal{V}}$. The distinction here to the case of balls is that some spheres in $\partial \mathcal{V}$ may lie inside the balls corresponding to distict spheres in $\mathcal{V}$.

We can now prove the following crucial lemma.

\begin{lem} \label{boundgen}
Let $(M,d)$ be well separable with constant $\chi$ and extension $(\tilde{M},\tilde{d})$. Let $\epsilon, \delta \in (0,1)$, $t \geq 0$ and $p = \left\lceil \frac{2 \chi}{\epsilon \delta} \right\rceil$. Suppose that
\begin{enumerate}[\normalfont (1)]
\item $\nu$ is a finite measure on $M$,
\item $F \subseteq M$ is finite and $\nu(F) > \delta \nu(M)$,
\item $\mathcal{U}_1,...,\mathcal{U}_p$ is a stack over $F$ with $\rmin{\mathcal{U}_i} > \, 2 \, \rmax{\mathcal{U}_{i-1}}$ and $\rmin{\mathcal{U}_1} > 2t$, and
\item $\nu(\partial_tB) > \epsilon \, \nu(B)$ for each $B \in \bigcup_i \mathcal{U}_i$
\end{enumerate}
then there is some integer $k \geq 2$ and a subcollection $\mathcal{V} \subseteq \bigcup_{i \geq k} \mathcal{U}_i$ such that
\begin{enumerate}[\normalfont (i)]
\item $\partial\mathcal{V}$ is well-separated and
\item $\nu \left(F \cap \bigcup_{B \in \mathcal{V}} \partial_{2r} B \right) >  \frac{1}{2} \, \nu(F)$, where $r = \rmax{\mathcal{U}_{k-1}}$.
\end{enumerate}
\end{lem}

\begin{proof}
We follow \cite{Hoch1}. Wlog $\nu(M) = 1$. We work up recursively from $l = 0$ and in stage $l$ we produce a collection $\mathcal{V} \subseteq \bigcup_{i > p - l} \mathcal{U}_i$ with $\partial \mathcal{V}$ well-separated and 
$$\nu \left( \bigcup_{B \in \mathcal{V}} \partial_{2r(l)} B \right) 
\geq \nu \left(\bigcup_{B \in \mathcal{V}} \partial_{t} B \right)
\geq 	\frac{\epsilon\delta}{2\chi} l$$ 
where $r(l) = \rmax{\mathcal{U}_{p-l}}$.

For $l=0$ take $\mathcal{V} = \emptyset$. Assume we have completed stage $l$. If
$$\nu \left( F \cap \bigcup_{B \in \mathcal{V}} \partial_{2r(l)} B \right) > \frac{1}{2} \, \nu(F)$$
then take $k = p - l + 1$ and we are done. Otherwise let $E = F \setminus \bigcup_{B \in \mathcal{V}} \partial_{2r(l)} B$ and note that $\nu(E) \geq \frac{1}{2} \nu(F) \geq \frac{\delta}{2}$. By the previous lemma we may choose a well-separated subcollection of balls $\mathcal{U}' \subseteq \mathcal{U}_{p-l}$ with $\nu \left( \mathcal{U}' \right) > \frac{\delta}{2 \chi}$. As $2t < r(l)$ it follows from (4) that
$$\nu\left(\bigcup_{B \in \mathcal{U}'} \partial_t B \right) > \epsilon \, \nu\left(\bigcup_{B \in \mathcal{U}'} B \right) \geq \frac{\epsilon\delta}{2 \chi}.$$
The centres of each $B \in \mathcal{U}'$ are strictly more than $2 \, r(l) > 4 \, r(l+1)$ from each element of $\partial \mathcal{V}$. Firstly this ensures $\partial \mathcal{V} \cup \partial \mathcal{U}'$ is well-separated. Secondly it means that
\begin{align*}
\nu\left( \bigcup_{B \in \mathcal{V} \cup \mathcal{U}'} \partial_{2r(l+1) } B \right) 	
&=		\nu\left( \bigcup_{B \in \mathcal{V}} \partial_{2r(l+1) } B \right)
		+ 
		\nu\left( \bigcup_{B \in \mathcal{U}'} \partial_{2r(l+1) } B \right) \\
&\geq	\nu\left( \bigcup_{B \in \mathcal{V}} \partial_{t} B \right)
		+ 
		\nu\left( \bigcup_{B \in \mathcal{U}'} \partial_{t} B \right)	\\
&> 		\frac{\epsilon\delta}{2\chi} (l+1)
\end{align*}
and so we can complete the recursive step by adding $\mathcal{U}'$ to $\mathcal{V}$.

The process must terminate by stage $l = p-1$, ensuring $k \geq 2$. If it does not then we can complete stage $p-1$. Then, as above, we see that 
$$\nu\left( \bigcup_{B \in \mathcal{V} \cup \mathcal{U}'} \partial_{t} B \right)  > \frac{\epsilon\delta}{2\chi} p \geq 1.$$
\end{proof}

The reason Lemma \ref{boundgen} is crucial is that we are going to repeatedly apply it to produce a series of collections of thickened boundaries, each containing a not insignificant portion of a finite set $F$, and then seek to apply the following property.

\begin{defn}
Let $(M,d)$ have extension $(\tilde{M},\tilde{d})$. We say $(M,d)$ has \emph{finite intersection dimension} (with respect to $\tilde{M}$) if there is a positive integer $\kappa$ and an $R > 1$ such that given
\begin{enumerate}[\normalfont (a)]
\item $t(1),...,t(\kappa) \geq 1$,
\item $r(1),...,r(\kappa)$ such that each $r(i) \geq t(1)...t(i)R$ and
\item points $x_1,...,x_\kappa \in M$ such that $x_i \in \bigcap_{j < i} \partial_{t(j)} B_{r(j)}(x_j)$ for all $i \leq \kappa$
\end{enumerate}
then $\bigcap_{i=1}^\kappa \partial_{t(i)} B_{r(i)}(x_i) = \emptyset$. In this case, we say that $(M,d)$ has \emph{intersection dimension $\kappa$ at scale $R$}.
\end{defn}

Note that if $(\tilde{M},\tilde{d})$ has intersection dimension $\kappa$ at scale $R$ with respect to $(\tilde{M},\tilde{d})$ then so too do all its restrictions. 

It is important to note here that the intersection dimension of a space is a minor reformulation of the coarse dimension defined by Hochman in \cite{Hoch1}, and uses a different notion of boundary. The two quantities are in fact the same in that paper's setting. The reason for using the name `intersection dimension' is simply to avoid potential confusion with another quantity called the coarse dimension, used in coarse geometry, as it is not clear there is a connection between the two. 

We can now give a new proof of Theorem 4.4 in \cite{Hoch1}, with a slight improvement in the bound for the height of the stack required.

\begin{theorem} \label{maintech}
Let $(M,d)$ be well separable with constant $\chi$ and have intersection dimension $\kappa \in \mathbb{N}_0$ at scales $R > 1$ with respect to an extension $(\tilde{M},\tilde{d})$. Let $0 < \epsilon, \delta < 1$. Then the following holds for some positive integer $q \leq \kappa \left(\frac{2 \sqrt{2} \chi}{\epsilon \delta} \right)^\kappa (\sqrt{2})^{\kappa^2}$. Suppose that
\begin{enumerate}[\normalfont (1)]
\item $\nu$ is a finite measure on $M$,
\item $F \subseteq M$ is finite,
\item $\mathcal{U}_1,...,\mathcal{U}_q$ is a stack over $F$ with
\begin{enumerate}[\normalfont (a)]
\item $\rmin{\mathcal{U}}_i > 2(\rmax{\mathcal{U}_{i-1}})^2$,
\item $\rmin{\mathcal{U}}_1 > 7\max{(t,R)}$,
\end{enumerate}
\item $\nu(\partial_tB) > \epsilon \, \nu(B)$ for each $B \in \bigcup_i \mathcal{U}_i$.
\end{enumerate}
Then $\nu(F) \leq \delta \, \nu(M)$.
\end{theorem}

\begin{proof}
Suppose for a contradiction that $\nu(F) > \delta \nu(M)$. Let $F_0 = F$, $p_i = \left\lceil \frac{2^{i+1} \chi}{\epsilon \delta} \right\rceil$, $q_\kappa = 0$ and set $q_i = p_i(1+q_{i+1})$ for each $0 \leq i \leq \kappa-1$. In particular $q = q_0$.

The idea behind this proof is to first apply Lemma \ref{boundgen} to find a collection of (thickened) spheres containing at least half the mass of $F$. We will then do this again with the portion of $F$ inside the first collection of spheres to produce a second collection (with centres inside spheres from the previous one) that contains at least a quarter of the mass in $F$. We will continue in this fashion until we have $\kappa$ such collections, with the last containing at least $2^{-\kappa}$ of a mass of $F$. By taking care to control the radii of the spheres at each stage we will ensure that any point in this portion of $F$ must lie in a sequence of $\kappa$ thickened spheres satisfying the definition of the intersection dimension, forcing a contradiction since the intersection of any such sequence must be empty. 

More precisely, we construct a sequence of sets $F = F_0 \supset F_1 \supset ... \supset F_\kappa$ and select positive integers $n_1,...,n_\kappa$ such that $1 \leq n_i \leq p_{i-1} - 1$ where for each $1 \leq i \leq \kappa$ we have $\nu(F_i) \geq \frac{1}{2} \nu(F_{i-1})$ and 
$$F_i = F_{i-1} \cap \bigcup_{B \in \mathcal{V}_i} \partial_{t(i)} B$$
with $\mathcal{V}_i$ being a subcollection of balls centred in $F_{i-1}$ from the stack
$\mathcal{U}_{N_i+ q_i + 1},...,\mathcal{U}_{N_{i-1}+q_{i-1}}$ and for which $\partial \mathcal{V}_i$ is well-separated. Here $N_i = \sum_{j=1}^i n_j(1+q_j)$ and $t(i) = 2 \, \rmax{\mathcal{U}_{N_i}}$. Note that our assumptions ensure that $N_{i-1} + 1 \leq N_i \leq  N_i+q_i + 1 \leq N_{i-1}+q_{i-1}$. In particular if $i < j$ then $N_j + q_j \leq N_i + q_i$.

\begin{flushleft}
	\textbf{How the sequences force a contradiction:}
\end{flushleft}

From these conditions we are able to deduce that 
$$\nu(F_\kappa) \geq \frac{1}{2^\kappa} \nu(F) \geq \frac{\delta}{2^\kappa} \nu(M) >0$$
and so in particular $F_\kappa$ is non-empty. 

Let
$$x \in F_\kappa = F \cap \bigcap_{i=1}^\kappa \bigcup_{B \in \mathcal{V}_i} \partial_{t(i)}B.$$
By definition of $F_\kappa$ there exist $x_1,...,x_\kappa$ and $r(1),...,r(\kappa)$ such that each $x_i \in F_{i-1}$ and $x \in \bigcap_{i=1}^\kappa \partial_{t(i)} B_{r(i)}(x_i)$. Suppose $i < j$ then $x_j \in F_{j-1} \subseteq F_i \subseteq \bigcup_{B \in \mathcal{V}_i} \partial_{t(i)} B$. Since the collection $\partial\mathcal{V}_i$ is well-separated its elements are a distance at least 
\begin{align*}
	\rmin{\mathcal{V}_i} 
	\geq \rmin{\mathcal{U}_{N_i + q_i +1}} 
&> 	(\rmax{\mathcal{U}_{N_i + q_i}})^2 \\
&>	 7 \, \rmax{\mathcal{U}_{N_i + q_i}} \\
&>	 2t(i) + t(j) + \rmax{\mathcal{U}_{N_{j-1} + q_{j-1}}}
\end{align*}
apart, where we have applied properties (a) and (b). Since $x_j$ lies within distance $t(i)$ of some element of $\partial\mathcal{V}_i$ (which is well separated) this inequality means the ball of radius $t(j) + \rmax{\mathcal{U}_{N_{j-1} + q_{j-1}}}$ about $x_j$ can intersect at most one of the thickened spheres
$\lbrace \partial_{t(i)} B : B \in \mathcal{V}_i \rbrace$. Since $\rmax{\mathcal{V}_j} \leq \rmax{\mathcal{U}_{N_{j-1} + q_{j-1}}}$ and $\partial_{t(i)} B_{r(i)}(x_i) \cap \partial_{t(j)} B_{r(j)}(x_j)$ is non-empty we see that this sphere is $\partial_{t(i)} B_{r(i)}(x_i)$ and hence $x_j \in \partial_{t(i)} B_{r(i)}(x_i)$. Next note that given $1 \leq i \leq \kappa$
\begin{align*}
r(i) \geq \rmin{\mathcal{U}_{N_i + q_i +1}} > \rmin{\mathcal{U}_{N_i +1}} > 2(\rmax{\mathcal{U}_{N_i}})^2 \geq t(i) \, \rmax{\mathcal{U}_{N_{i-1}+1}} 
\end{align*}
and by recursion
\begin{align*}
r(i) &> t(i)t(i-1)...t(2) \, \rmax{\mathcal{U}_{N_{1}+1}} \\
&> t(i)t(i-1)...t(2)t(1) \, \rmax{\mathcal{U}_{N_{1}}} \\
&> t(i)t(i-1)...t(2)t(1)R.
\end{align*}
This means that the $x_i$, $r(i)$ and $t(i)$ satisfy the conditions in the definition of the intersection dimension and so $\bigcap_{i=1}^\kappa \partial_{t(i)} B_{r(i)}(x_i) = \emptyset$, a contradiction.

\begin{flushleft}
	\textbf{Constructing the sequences:} 
\end{flushleft}

All that remains is to show such collections $\mathcal{V}_i$ and integers $n_i$ exist, and for this we will use Lemma \ref{boundgen}. Given these up to a certain $0 \leq i \leq \kappa-1$ we produce the $i+1$ set as follows: consider the stack $\mathcal{U}_{N_i + 1}, ..., \mathcal{U}_{N_i + q_i}$
over $F$. This can clearly be restricted to a stack $\mathcal{U}_{N_i + 1}', ..., \mathcal{U}_{N_i + q_i}'$ over $F_i$ by simply taking the balls with centres in $F_i$, and it inherits all the radii growth conditions from the original stack. In particular, we may apply Lemma \ref{boundgen} to the stack $\lbrace \mathcal{U}_{N_i + j(1+q_{i+1})}' \rbrace_{j=1}^{p_i}$ to find $1 \leq n \leq p_i -1$ and find a subcollection $\mathcal{V} \subseteq \bigcup_{n + 1 \leq j \leq p_i} \mathcal{U}_{N_i + j(1+q_{i+1})}$ for which (if we take $n_{i+1} = n$)
$$\nu{\left( F_i \cap \bigcup_{B \in \mathcal{V}} \partial_{t(i+1)}B \right)} > \frac{1}{2} \nu(F_i)$$
since $N_{i+1} = N_i + n_{i+1}(1+q_{i+1})$. By noting the range of $j$, we see that $\mathcal{V}$ consists of balls from the stack $\mathcal{U}_{N_{i+1}+ q_{i+1} + 1},...,\mathcal{U}_{N_{i}+q_{i}}$, and so we may take $\mathcal{V}_{i+1} = \mathcal{V}$.

\begin{flushleft}
	\textbf{Proving the bound on $q$:}
\end{flushleft}

To get the bound on $q$ observe that
\begin{align*}
q = q_0 
&= 		p_0(1+p_2(1+...(1 + p_{\kappa-1})...) 
= 		\sum_{i=0}^{\kappa-1} \prod_{j=0}^i p_j \\
&\leq	
\sum_{i=0}^{\kappa-1} \left( \frac{2\chi}{\epsilon\delta} \right)^i 	 2^{\sum_{j=0}^i j}
\leq \kappa \left( \frac{2 \sqrt{2} \chi}{\epsilon\delta} \right)^{\kappa-1} 	 \sqrt{2}^{(\kappa-1)^2}
\end{align*}
which completes the proof.
\end{proof}

With this result in hand we are now ready to show that 
\begin{align*}
\frac{\sum_{g \in B_k \triangle \sigma B_k} \omega_g}{\sum_{g \in B_k} \omega_g} \rightarrow 0 \qquad \text{a.s.}
\end{align*}
under a further assumption on the metric structure.

We say that a metric space $(M,d)$ is \emph{voidless} if for all $x \in M$ and $r > 0$, every closed ball $B \subset M$ such that $B \cap \lbrace y : d(x,y) < r \rbrace \neq \emptyset$ and $B \cap \lbrace y : d(x,y) > r \rbrace \neq \emptyset$ satisfies $B \cap \lbrace y : d(x,y) = r \rbrace \neq \emptyset$. In particular, note that if a metric space is such that every closed ball is path connected then the intermediate value theorem ensures it is voidless. 

\begin{lem} \label{bdybnd}
Let $(G,d)$ be a group metric space and suppose that there is a voidless right invariant group metric space $(\tilde{G},\tilde{d})$ for which $G \leq \tilde{G}$ and $(\tilde{G},\tilde{d})$ extends $(G,d)$. Then for any closed ball $B \subset G$ and $\sigma \in G$ we have $B \triangle \sigma B \subseteq \partial_t B$ where $t = d(\sigma,0)$.
\end{lem}

\begin{proof}
Let $g \in G$, then $d(\sigma^{-1} g, g) = d(\sigma^{-1},0) = d(\sigma,0)$ i.e. $\sigma^{-1} g \in B_t(g)$. Suppose $g \not\in \partial_t B$, then $\tilde{d}(g,\partial \tilde{B}) > t$ (here $\tilde{B}$ denotes the ball in $\tilde{G}$ with the same centre and radius as $B$). Hence $\tilde{B}_t(g)$ does not intersect $\partial \tilde{B}$. Since $\tilde{G}$ is voidless it follows that either $B_t(g) \subseteq B$ or $B_t(g) \subseteq B^c$. Therefore, if $g \in B$ then $\sigma^{-1} g \in B$ and if $g \in B^c$ then $\sigma^{-1} g \in B^c$ and so either $g \in B \cap \sigma B$ or $g \not \in B \cup \sigma B$, i.e. $g \not\in \partial_t B$ implies $g \not\in B \triangle \sigma B$.
\end{proof}

So, under the conditions of the lemma given $\sigma$ it is enough for us to prove that 
\begin{align*}
\frac{\sum_{g \in \partial_t B_k} \omega_g}{\sum_{g \in B_k} \omega_g} \rightarrow 0 \qquad \text{a.s.}
\end{align*}
with $t = d(\sigma,0)$.

\begin{defn}
Given a sequence $e \in B_1 \subset B_2 \subset ...$ of finite subsets of a countable group $G$ we say that it has the \emph{multiplicative doubling property}  (MDP) if there exists constants $D > 0$ and $K \in \mathbb{N}$ such that $|B_kB_k| \leq D |B_k|$ for all $k \geq K$.
\end{defn}

\begin{theorem} \label{bdycon}
Let $(G,d)$ be a countable group metric space which acts non-singularly on the probability space $(X,\mu)$, is well separable and such that the sequence of integer balls $(B_k)$ has the MDP. Suppose that there is a right invariant group metric space $(\tilde{G},\tilde{d})$ for which $G \leq \tilde{G}$ and $(\tilde{G},\tilde{d})$ extends $(G,d)$. Furthermore suppose $(G,d)$ has finite intersection dimension with respect to $(\tilde{G},\tilde{d})$. Then for all $t > 0$
\begin{align*}
\lim_{k \rightarrow \infty} \frac{\sum_{g \in \partial_t B_k} \omega_g}{\sum_{g \in B_k} \omega_g} = 0 \qquad \text{a.s..}
\end{align*}
\end{theorem}

\begin{proof}
Here we return to the standard approach laid out in \cite{Hoch1}. Let $\chi$ be the constant of well separability, the intersection dimension be $\kappa$ at scales $R$ and $D$ be the multiplicative doubling constant.

Suppose for a contradiction that for some $\epsilon > 0$ 
$$\limsup_{k \rightarrow \infty} F_k(x) > \epsilon$$
on some set $A_0$ with positive measure, where $$F_k = \frac{\sum_{g \in \partial_t B_k} \omega_g}{\sum_{g \in B_k} \omega_g}.$$

Now we construct a sequence of integers $r_1^- < r_1^+ < r_2^- < r_2^+ < ...$ and sets $A_0 \supset A_1 \supset A_2 \supset ...$ as follows: we first let $r_0^+ = 0$ and ensure $r_1^- > 7 \max(t,R)$. Then for each $i \geq 1$ given $r_{i-1}^+$ and $A_{i-1}$ with $\mu(A_{i-1}) > \frac{1}{2}(1+\frac{1}{i})\mu(A_0)$ we take $r_{i}^- > 2 (r_{i-1}^+)^2$ and let
$$A_{i} = \left\lbrace x \in A_{i-1}: \max_{r_{i}^- \leq j \leq r_{i}^+} F_j(x) > \epsilon \right\rbrace$$
where $r_{i}^+$ is chosen large enough to ensure that $\mu(A_{i}) > \frac{1}{2}(1+\frac{1}{i+1})\mu(A_0)$. In particular these properties ensure that the set $A = \bigcap_{i=0}^\infty A_i$ has measure at least $\frac{1}{2}\mu(A_0) > 0$, and that the $r_i^\pm$ satisfy the radii growth conditions for Theorem \ref{maintech}. We will use this latter property so show that we must have $\mu(A) = 0$, giving the contradiction.

Fix $\delta > 0$ and take $q = q(\chi,\kappa,\epsilon,\delta)$ as in Theorem \ref{maintech}. Fix $k > r_q^+ + t$ large enough to employ the MDP. Observe that 
$$\mu(A) = \frac{1}{|B_k|} \int_X \sum_{g \in B_k} \hat{g} \ind_A \, d\mu.$$
If we fix $x \in X$ then we may define a measure $\nu = \nu_{x,k}$ on $B_k^2 = B_kB_k$ by 
$\nu(E) = \sum_{g \in E} \omega_g(x)$. Note that since $B_k^2$ is contained by $G$ we are able to take it as $M$ in Theorem \ref{maintech} (it inherits well-separability and the finite intersection dimension properties from $G$). In addition for almost every such $x$ the measure $\nu$ is finite, since $B_k^2$ is, so it will suffice for us to consider only these $x$. Let $F = \lbrace g \in B_k : gx \in A \rbrace$, so that $$\nu(F) = \sum_{g \in B_k} \ind_A(gx)\omega_g(x) = \sum_{g \in B_k} \hat{g} \ind_A(x).$$

We can construct a stack over $F$ as follows. If $h \in F$ then $hx \in A$ and so for each $1 \leq i \leq q$ there is $r_i^- \leq m = m(i,h) \leq r_i^+$ for which
$$\sum_{g \in \partial_t B_{m}} \omega_g(hx) > \epsilon \sum_{g \in B_m} \omega_g(hx).$$
Note that as $B_r(g) = B_r g$, $h \in B_k$, $m \leq r_q^+$, and $k > r_q^+ + t$ we have $B_m(h) \subseteq B_k^2$ and $\partial_t B_m(h) \subseteq B_{m+t}(h) \subseteq B_k^2$. Hence $\nu(\partial_t B_m(h)) > \epsilon \, \nu(B_m(h))$. It follows that given $1 \leq i \leq q$ we can let $\mathcal{U}_i = \lbrace B_{m(i,h)}(h) : h \in F \rbrace$, and this stack satisfies all the requirements of 
Theorem \ref{maintech}.

Applying the theorem it follows that $\nu(F) \leq \delta \nu(B_k^2)$ for a.e.\ $x \in X$ and we may apply the multiplicative doubling condition to see that
$$\mu(A) \leq \frac{1}{|B_k|} \int_X \delta \sum_{g \in B_k^2} \omega_g \, d\mu = \frac{|B_k^2|}{|B_k|} \delta \leq D \delta.$$
Since $\delta > 0$ was arbitrary, we are done. 
\end{proof}

\subsection{The maximal inequality}

In this section we follow the exposition given in \cite{Fel1} to prove the maximal inequality. For the interested reader, \cite{Fel1} also gives a concise account of the various authors who contributed to the approach.

A geometrical assumption thought to be essential to the maximal inequality, see \cite{Hoch1}, is the Besicovitch covering property.

\begin{defn} \label{BCP}
A metric space $(M,d)$ has the \textit{Besicovitch covering property} (BCP) if there is a constant $C > 0$ such that for any finite set $E \subset X$ and any carpet $\, \mathcal{U}$ over $E$ we have a subcollection $\mathcal{V} \subseteq \mathcal{U}$ for which
$$\ind_E \leq \sum_{U \in \mathcal{V}} \ind_U \leq C.$$ 
In this situation we say $(M,d)$ has the BCP with constant $C$. A carpet satisfying the second inequality above is said to have multiplicity $C$.
\end{defn}

This property plays a crucial role in the following lemma.

\begin{lem} \label{maxlem}
Let $G$ be a countable group and $(G,d)$ be a right invariant metric space which satisfies the BCP with constant $C$. For each function $a \in l^1(G)$ and $k \in \mathbb{N}$ let $s_k a (h) = \sum_{g \in B_k} a(gh)$ for all $h \in G$. Given $k \in \mathbb{N}$ and $a,b \in l^1(G)$ with $b \geq 0$ the set $H = H_k(a,b) = \bigcup_{i=1}^k H^{(i)}$, where $H^{(i)} = \lbrace h \in G : s_i a(h) > \epsilon s_i b(h) \rbrace$, satisfies
$$\|a\|_1 \geq \epsilon C^{-1} \sum_{h \in H} b(h).$$  
\end{lem}

\begin{proof}
Let $E \subset G$ be finite and consider that for each $h \in E \cap H$ there is a $1 \leq m(h) \leq k$ for which $h \in H^{(m(h))}$. The collection of balls $B_{m(h)}(h) = B_{m(h)}h$ with $h \in  E \cap H$ describes a carpet over $E \cap H$ and hence by the BCP we can find a set $F \subset  E \cap H$ for which
$$\ind_{E \cap H} \leq \sum_{h \in F} \ind_{B_{m(h)}(h)} \leq C.$$
It follows that
\begin{align*}
\sum_{h \in E \cap H} b(h) \leq \sum_{h \in F} \sum_{g \in B_{m(h)}(h)} &b(g) = \sum_{h \in F} s_{m(h)}b(h) \\
&< \epsilon^{-1} \sum_{h \in F} s_{m(h)}a(h) = \epsilon^{-1} \sum_{h \in F} \sum_{g \in B_{m(h)}(h)} a(g)  \leq  \epsilon^{-1}C \|a\|_1
\end{align*}
and as $E$ was arbitrary the result follows.
\end{proof}

Lemma \ref{maxlem} combines with the multiplicative doubling property to give the maximal inequality.

\begin{theorem}[The maximal inequality]
Let $G$ be a countable group and $(G,d)$ be a right invariant metric space which has the BCP with constant $C$ and suppose the sequence of integer balls $(B_k)$ has the MDP with constant $D$. Then for any $f \in L^1$ and $\epsilon > 0$
$$\mu\left( \sup_{k \geq 1} \left| \frac{\sum_{g \in B_k} \hat{g}f }{\sum_{g \in B_k} \hat{g}1} \right| > \epsilon \right) \leq \frac{CD}{\epsilon} \| f \|_1.$$
\end{theorem}

\begin{proof}
For convenience let
$$F_k = F_k(x) = \frac{\sum_{g \in B_k} \hat{g}|f| }{\sum_{g \in B_k} \hat{g}1}.$$
Fix $K \in \mathbb{N}$ large enough to employ the MDP, it is enough to show that 
$$\mu\left( \max_{1 \leq k \leq K}  F_k  > \epsilon \right) \leq \frac{CD}{\epsilon} \| f \|_1.$$ 
We consider $f \geq 0$ without loss of generality. Now we fix a typical $x \in X$ and seek to apply Lemma \ref{maxlem} with 
$a_x(h) = \ind_{B_K^2}(h) [\hat{h}f(x)] = \ind_{B_K^2}(h) f(hx)\omega_h(x)$ and $b_x(h) = \ind_{B_K^2}(h) \omega_h(x)$. Observe that if $h \in B_k$ then since $\omega_{gh}(x) = \omega_g(hx)\omega_h(x) \text{ a.e.}$
$$s_n a_x(h) = \sum_{g \in B_k} \ind_{B_K^2}(gh) f(ghx)\omega_{gh}(x) = \sum_{g \in B_k} f(ghx)\omega_{gh}(x) = \omega_h(x) \sum_{g \in B_k} \hat{g}f(hx)$$
and similarly
$$s_k b_x(h) = \omega_h(x) \sum_{g \in B_k} \hat{g}1(hx).$$
In particular, for almost every $x \in X$ we have $s_k a_x(h) > \epsilon s_k b_x(h)$ if and only if $F_k(x) > \epsilon$. Let $Y = \lbrace x \in X : \max_{1 \leq k \leq K}  F_k(x) > \epsilon \rbrace$ and $H_x = H_K(a_x,b_x)$ from Lemma \ref{maxlem}. Then $gx \in Y$ if and only if $g \in H_x$, and hence
\begin{align*}
\mu(Y) = \frac{1}{|B_K|} \int \sum_{g \in B_K} \hat{g}\ind_Y \, d\mu 
	   &= \frac{1}{|B_K|} \int \sum_{g \in H_x} \ind_{B_K}(g) \omega_g \, d\mu \\
	   &\leq \frac{C}{\epsilon|B_K|} \int \|a_x\|_1 \, d\mu 
	   = \frac{C}{\epsilon|B_K|} \int \sum_{g \in B_K^2} \hat{g}f \, d\mu 
	   = \frac{C |B_K^2|}{\epsilon|B_K|} \|f\|_1 
\end{align*}
since $\ind_{B_K}(g) \omega_g \leq b_x(g)$. The result then follows from the multiplicative doubling condition.
\end{proof}

\subsection{Completion of the proof}

We are now able to prove the ergodic theorem.

\begin{theorem}[The ergodic theorem] \label{genergthm}
Let $G$ be a countable group, equipped with a metric $d$, which has an ergodic non-singular action on the standard probability space $(X,\mu)$. Suppose that:
\begin{enumerate}[\normalfont (i)]
\item $(G,d)$ is well separable,
\item the sequence of integer balls $(B_k)$ has the multiplicative doubling property,
\item there is a voidless right invariant group metric space $(\tilde{G},\tilde{d})$ for which $G \leq \tilde{G}$ and $(\tilde{G},\tilde{d})$ extends $(G,d)$,
\item $(G,d)$ has finite intersection dimension with respect to $(\tilde{G},\tilde{d})$ and
\item $(G,d)$ has the Besicovitch covering property
\end{enumerate}
then for every $f \in L^1(\mu)$
\begin{align*}
\lim_{n \rightarrow \infty} \frac{\sum_{g \in B_k} \hat{g}f}{\sum_{g \in B_k} \hat{g}1} = \int f \, d\mu \qquad \text{almost everywhere.}
\end{align*}
\end{theorem}

\begin{proof}
Let $C$ and $D$ be the Besicovitch and doubling constants. We have already seen that the set
$$S = \lbrace c + h - \hat{\sigma}h : c \in \mathbb{R}, \sigma \in G, h \in L^\infty \rbrace$$
is dense in $L^1$ and that given $\sigma \in G$
\begin{align*}
\frac{\sum_{g \in B_k} \hat{g}(c+h - \hat{\sigma}h)}{\sum_{g \in B_k} \hat{g}1} \rightarrow c \qquad \text{a.e.}
\end{align*} 
must occur if
$$\frac{\sum_{g \in B_k \triangle \sigma B_k} \omega_g}{\sum_{g \in B_k} \omega_g} \rightarrow 0 \qquad \text{a.e.}.$$
This latter condition follows from first using (iii) to apply Lemma \ref{bdybnd} and then using (i), (ii) and (iv) to apply Theorem \ref{bdycon} with $t = d(0,\sigma)$.

Now, given $f \in L^1$ we may choose a sequence $f_m = c_m + h_m - \hat{\sigma}_m h_m \in S$ such that $\| f - f_m \|_1 \leq \frac{1}{m}$ for all $m \geq 1$. In particular $c_m = \int f_m \, d\mu \rightarrow \int f \, d\mu$. 

Fix $\epsilon > 0$. By applying (ii) and (v) via the maximal inequality to $f - f_m$ we see that
$$\mu\left(\limsup_{k \to \infty} \left\lvert \frac{\sum_{g \in B_k} \hat{g}f}{\sum_{g \in B_k} \hat{g}1} - c_m \right\rvert > 2\epsilon \right) \leq \frac{2CD}{m\epsilon}$$
and hence, if we choose $m$ large enough for $|c_m - \int f \, d\mu| < \epsilon$ then
$$\mu\left(\limsup_{k \to \infty} \left\lvert \frac{\sum_{g \in B_k} \hat{g}f}{\sum_{g \in B_k} \hat{g}1} -  \int f \, d\mu\right\rvert > \epsilon \right) \leq \frac{2CD}{m\epsilon}$$
for all $m$ sufficiently large. Since $\epsilon > 0$ was arbitrary the result follows.
\end{proof}

\section{An ergodic theorem for $\mathrm{H}^n$-actions}

In this part we will show that the ergodic theorem holds for the discrete Heisenberg group by checking it satisfies the conditions of Theorem \ref{genergthm} when equipped with the metric used by Le Donne and Rigot in \cite{LeDRig1}. In the paper they showed the metric $d$, defined below, has the Besicovitch covering property. Therefore it is sufficient for us to address properties (i)-(iv) in the theorem.

\subsection{Setup} \label{Heissetup}

We shall define the \textit{$n$-dimensional continuous Heisenberg group}, $\mathbb{H}^n$, as follows. As a set take $\mathbb{H}^n = \mathbb{C}^n \times \mathbb{R}$ and equip it with the multiplication given by 
\begin{align*}
(z,\tau) \cdot (w,\sigma) = \left(z + w, \tau + \sigma + \frac{1}{2} \im{\langle z, w \rangle} \right)
\end{align*}
where $z,w \in \mathbb{C}^n$, $\tau, \sigma \in \mathbb{R}$ and the inner product is the standard one on $\mathbb{C}^n$, given by $\langle z, w \rangle = \sum_{j=1}^n \overline{z_j} w_j$. This is essentially the same realisation as that used by Le Donne and Rigot in \cite{LeDRig1} except we are using complex coordinates.

The \textit{$n$-dimensional discrete Heisenberg group} $\mathrm{H}^n$ is the discrete subgroup generated by the elements of the form $(e_j,0)$ or $(ie_j,0)$ where $e_j$ is a standard basis vector of $\mathbb{R}^n$. As a set
\begin{align*}
\mathrm{H}^n = \lbrace (z,\tau) \in \mathbb{H}^n : z \in \mathbb{Z}^n + i \mathbb{Z}^n , \tau \in \tfrac{1}{2}\langle \re{z}, \im{z} \rangle + \mathbb{Z} \rbrace.
\end{align*}
We will be taking $G = \mathrm{H}^n$ and $\tilde{G} = \mathbb{H}^n$ in Theorem \ref{genergthm}.

For each $\lambda > 0$ there is a dilation map $\delta_\lambda : \mathbb{H}^n \to \mathbb{H}^n$ given by
$$\delta_\lambda(z,t) = (\lambda z, \lambda^2 \tau).$$
Each $\delta_\lambda$ is an automorphism of $\mathbb{H}^n$.

To describe the range of the sums in the ergodic theorem we will be using the balls of the metric given by 
$$d\left( p, q \right) = \inf \left\lbrace r > 0 : \delta_{1/r}(pq^{-1}) \in B_{eucl} \right\rbrace$$
where $B_{eucl}$ denotes the closed euclidean unit ball on $\mathbb{C}^n \times \mathbb{R}$. This is the right invariant version of the metric given in \cite{LeDRig1} with $\alpha = 1$. It is \textit{one-homogeneous} with respect to the dilation, i.e. for all $\lambda > 0$ and $p,q \in \mathbb{H}^n$ we have
$d(\delta_\lambda p, \delta_\lambda q) = \lambda \, d(p,q).$ By considering the case $q = 0$ and using right invariance it is not difficult to show that for $p = (z,\tau)$ and $q = (w,\sigma)$
\begin{align}
d(p,q) \leq r \qquad \Longleftrightarrow \qquad \frac{\|z-w\|^2}{r^2} + \frac{\left( \tau - \sigma - \frac{1}{2} \im{\langle z, w \rangle} \right)^2}{r^4} \leq 1 \label{ball}
\end{align}
and 
\begin{align}
d(p,q) = r \qquad \Longleftrightarrow \qquad \frac{\|z-w\|^2}{r^2} + \frac{\left( \tau - \sigma - \frac{1}{2} \im{\langle z, w \rangle} \right)^2}{r^4} = 1 \label{sphere}
\end{align}
where $\|\cdot\|$ is the euclidean norm on $\mathbb{C}^n$. In particular, taking $r = 1$ and $q = 0$ shows that the unit sphere of $d$ is exactly the Euclidean unit sphere, and similarly for the unit ball. This property is key to many of arguments to follow. It also follows from (\ref{sphere}) that
\begin{align}
d(p,q) = \frac{1}{\sqrt{2}}\left(\|z-w\|^2 + \sqrt{\|z-w\|^4 + 4 \left( \tau - \sigma - \frac{1}{2} \im{\langle z, w \rangle} \right)^2} \right)^{\frac{1}{2}}. \label{metric}
\end{align}
This explicit expression can be used show that $d$ is in fact a metric. In addition, as stated in \cite{LeDRig1}, $d$ induces the euclidean topology. $d$ therefore defines a (right) \textit{homogeneous distance} on $\mathbb{H}^n$, i.e. it induces the euclidean topology, is right invariant and one-homogeneous for the dilation.

Observe that we can use the dilations and right invariance to describe any ball in $\mathbb{H}^n$, explicitly for each $r > 0$ and $p \in \mathbb{H}^n$ the closed ball $B_r(p) = \delta_r(B_{eucl}) \cdot p$. Since the dilation is a linear map and right multiplication by $p$ is an affine map it follows that each ball is convex, in the euclidean sense.

It will be useful for us to note the following. Let $R_\theta$ be the $n \times n$ complex diagonal matrix with $R_\theta(j,j) = e^{i\theta_j}$ where $\theta = ( \theta_j )_{j=1}^n \in \mathbb{R}^n$. Then the maps
\begin{align} \label{isometries}
(z, \tau) \mapsto (\overline{z},-\tau) & & \text{and} & & (z, \tau) \mapsto (R_\theta z,\tau)
\end{align}
are isometries of $d$.

We will use $d$ to both denote the metric on $\mathbb{H}^n$ and its restriction to $\mathrm{H}^n$. 

We are now ready to start checking the conditions of Theorem \ref{genergthm} are satisfied. Since $d$ is right invariant and all the balls are euclidean convex (and so are path connected) this setup satisfies property (iii). It has already been mentioned that the central result of \cite{LeDRig1} is that $d$ satisfies the Besicovitch covering property on $\mathbb{H}^n$ and hence on $\mathrm{H}^n$, which covers property (v). Property (ii), the multiplicative doubling property, is a consequence of the following lemma. We will also use this lemma to in the proof that property (iv) holds.

\begin{lem} \label{sep-pt}
Let $(\mathbb{H}^n,d)$ be as above and $\rho > 0$. There exists $N(\rho) \in \mathbb{N}$ such that there are $N$ open balls of radius $\rho/2$ centred in $B_1(0)$ whose union covers $B_1(0)$. Consequently:
\begin{enumerate}[\normalfont (i)]
\item if $p_1,...,p_n \in B_1(0)$ and for all $i \neq j$ $d(p_i,p_j) > \rho$ then $n \leq N$, and
\item if $p_1,...,p_n \in B_1(0)$ with $n \geq kN$ for some $k \in \mathbb{N}$ then there is a subset $I \subset \lbrace 1,...,n \rbrace$ of size at least $k$ with $d(p_i,p_j) < \rho$ for all $i,j \in I$.
\end{enumerate}
\end{lem}

\begin{proof}
Since the metric $d$ induces the euclidean topology the closed unit ball $B_1(0) = B_{eucl}$ is compact, and the existence of such an $N$ follows from this compactness. Part (i) is due to the fact that if two points lie in the same ball in the cover then they are $< \rho$ apart, and part (ii) uses this along with the pigeon-hole principle.
\end{proof}

In particular, if we let $r > 0$, $p \in \mathbb{H}^n$ and $\rho = 1$ in Lemma \ref{sep-pt} then $B_r(p) = \delta_r(B_{eucl}) \cdot p$ can be covered by $N(1)$ balls of radius $\frac{r}{2}$ (simply dilate and translate those used to cover $B_{eucl}$). This means exactly that $(\mathbb{H}^n,d)$ has the \textit{metric} doubling property. 

\begin{cor}
The sequence $B_m = B_m(0) \cap \mathrm{H}^n$ has the multiplicative doubling property.
\end{cor}

\begin{proof}
Fix $0< r < \frac{1}{2} \inf{\lbrace d(p,q) : p,q \in \mathrm{H}^n, p \neq q \rbrace} = \frac{1}{2}$ and $$s = \sup{\lbrace d(p, \mathrm{H}^n) : p \in \mathbb{H}^n \rbrace} \leq \frac{1}{2} \sqrt{n + \sqrt{n^2 + 4}}.$$
Let $\nu$ be the right invariant Haar measure on $\mathbb{H}^n$ and $\rho = 2 \frac{m-s}{2m+r}$ where $m \in \mathbb{N}$ is taken sufficiently large to ensure $\rho \geq \frac{2}{3}$. Then by applying Lemma \ref{sep-pt} (between dilating by $(2m+r)^{-1}$ and $2m+r$) it follows that
\begin{align*}
|B_m^2| \nu(B_r(0)) 
= 		\nu\left(\bigcup_{p \in B_m^2} B_r(p) \right) 
&\leq 	\nu\left(B_{2m+r}(0) \right)\\
&\leq 	N(\rho) \, \nu(B_{m-s}(0)) \\
&\leq N \nu{\left( \bigcup_{p \in B_m} B_s(p) \right)} 
\leq N \nu{\left( B_s(0) \right)} |B_m|
\end{align*}
where $N = N(2/3)$. The result follows since balls with strictly positive radius have positive Haar measure.
\end{proof}

The remaining properties are (i), that $(\mathrm{H}^n,d)$ is well separable, and (iv), that $(\mathrm{H}^n,d)$ has finite intersection dimension with respect to $(\mathbb{H}^n,d)$. These require a bit more work, and are tackled in the next sections. From previous comments it is sufficient to show that $(\mathbb{H}^n,d)$ is well separable and has finite intersection dimension with respect to itself.

\subsection{Intersection dimension: the separation lemmas}

We start with the intersection dimension. Recall that in order to prove the intersection dimension is $\kappa$ we must show that given a sequence of points $p_1,...,p_m$ with $m \geq \kappa$ and thickened spheres about those points, with some conditions on the thickenings and radii, the intersection of these thickened spheres is empty. We will prove this in two stages. The first to repeatedly apply the principle that if, by increasing $m$, we can find a subsequence of arbitrary length with an additional property then we can replace the original sequence with this subsequence. The lemmas in this section will be used to impose these extra properties on the sequence. In the second stage we will use these to show that the resulting sequence of thickened spheres will have empty intersection if it is sufficiently long.

Given $p \in \mathbb{H}^n \setminus \lbrace 0 \rbrace$ let $\hat{p}$ be its unique dilate inside on the unit sphere, i.e. $\hat{p} = \delta_{1/\lambda} p$ where $\lambda = d(p,0) > 0$. We will call $\hat{p}$ the \textit{projection} (of $p$) onto the unit sphere.

The first lemma, below, will be used to show that if radii of the earlier spheres aren't too large compared to later ones, and $0$ is in their intersection, then their projections must be a fixed distance apart. This will allow us to assume that each radius is rather small compared to those preceding it.

\begin{lem}[Large scale separation] \label{LSS}
Let $p,q \in \mathbb{H}^n \setminus \lbrace 0 \rbrace$ with ${0 \in \tb{t}{r}{p} \cap \tb{\tilde{t}}{\tilde{r}}{q}}$ and $q \in \tb{t}{r}{p}$ where $t, \tilde{t} \geq 1$, $r \geq \tilde{r} \geq t \tilde{t} R$ and $R > 1$. Given $\epsilon \in (0,1)$ such that $\tilde{r} \geq \epsilon r$ there exists $\bar{R}(\epsilon) > 0$ such that if $R > \bar{R}$ then $$d(\hat{p},\hat{q}) \geq \frac{1}{2}\left(1 - \sqrt{1-\frac{\epsilon^2}{4}} \right) > 0.$$
\end{lem}

\begin{proof}
The triangle inequality ensures that $d(p,q) = r + s'$ and $d(p,0) = r + s$ for some $s, s'$ such that $|s|, |s'| \leq t$, and $d(q,0) = \tilde{r} + \tilde{s}$ for some $\tilde{s}$ with $|\tilde{s}| \leq \tilde{t}$. Therefore
\begin{align*}
\frac{r+s'}{r+s} = d(\hat{p},\delta_{\lambda}\hat{q}) \leq d(\hat{p},\hat{q}) + d(\hat{q},\delta_{\lambda}\hat{q})
\end{align*}
where $\lambda = \frac{\tilde{r}+\tilde{s}}{r+s}$. Note that
$$\frac{\epsilon}{2} \leq \frac{\epsilon-R^{-1}}{1+R^{-1}} \leq \lambda = \frac{\tilde{r}/r+\tilde{s}/r}{1+s/r} \leq \frac{1 + R^{-1}}{1 - R^{-1}}$$
for $R$ sufficiently large, depending on $\epsilon$.
Since $|1-\lambda|^2 \leq |1 - \lambda^2|$ it follows from (\ref{metric}) that 
\begin{align*}
d(\hat{q},\delta_{\lambda}\hat{q}) \leq \sqrt{|1 - \lambda^2|} d(0,\hat{q}) = \sqrt{|1 - \lambda^2|}.
\end{align*}
Therefore if $\lambda \leq 1$ then
\begin{align*}
d(\hat{p},\hat{q}) \geq 1 - \frac{2R^{-1}}{1 + R^{-1}} - \sqrt{1-\frac{\epsilon^2}{4}} \geq \frac{1}{2}\left(1-\sqrt{1-\frac{\epsilon^2}{4}} \right)
\end{align*}
for $R$ sufficiently large. Otherwise if $\lambda > 1$ then 
\begin{align*}
d(\hat{p},\hat{q}) \geq 1 - \frac{2R^{-1}}{1 + R^{-1}} - \sqrt{\frac{2R^{-1}}{1 - R^{-1}}} \geq \frac{1}{2}\left(1-\sqrt{1-\frac{\epsilon^2}{4}} \right)
\end{align*}
again for $R$ large enough.
\end{proof}

For the purposes of the remainder of this section it is useful to introduce a coordinate system on $\mathbb{H}^n$ which exploits the dilations and the fact that the unit sphere of $d$ is the Euclidean unit sphere. It is here that we are directly using properties of $(\mathbb{H}^n,d)$.

Given $p \in \mathbb{H}^n \setminus \lbrace 0 \rbrace$ let $\lambda_p = d(p,0) > 0$. Then $\hat{p} = \delta_{1/\lambda_p} p = (z_p,\tau_p)$ for some unique $z_p \in \mathbb{C}^n$ and $\tau_p \in \mathbb{R}$ with $\|z_p\|^2 + \tau_p^2 = 1$. In addition, using complex coordinates we have $\rho(p) = \left(\rho_i(p) \right)_{j = 1}^n \in S^n$, the Euclidean unit sphere, and $\phi(p) = \left(\phi_j(p) \right)_{j=1}^n \in (-\pi,\pi]^n$ such that $z_p = \left(\rho_j(p)\exp{[ i\phi_j(p) ]} \right)_{j=1}^n$. Given also $q \in \mathbb{H}^n \setminus \lbrace 0 \rbrace$ for each $1 \leq j \leq n$ let $\phi_j(p,q) \in \left[0,\pi \right)$ denote the magnitude of the angle between $\exp{[i\phi_p(j)]}$ and $\exp{[i\phi_q(j)]}$ in $\mathbb{C}$. 

By applying Lemma \ref{sep-pt} we will be able to assume that $\hat{p}_1,...,\hat{p}_m$ are close on the unit sphere, and Lemma \ref{LSS} will then allow us to assume that the radius of each sphere is small compared to the previous one. This when we will use the following small scale separation lemmas to narrow down the possible positions of $\hat{p}_1,...,\hat{p}_m$ relative to one another. 

\begin{lem}[Small scale separation 1] \label{smalltau}
Given any $\bar{\tau} \in (0,1)$ there exist $\bar{\rho},\bar{R},\bar{\phi},\bar{\epsilon} > 0$ for which the following holds. Let $p,q \in \mathbb{H}^n \setminus \lbrace 0 \rbrace$ with $q \in \tb{t}{r}{p}$ and $0 \in \tb{t}{r}{p} \cap \tb{\tilde{t}}{\tilde{r}}{q}$ where $t, \tilde{t} \geq 1$ and suppose $r \geq \tilde{r} \geq t \tilde{t} R$ for some $R > 1$. Suppose also that $\tilde{r} \leq \epsilon r$. If  $R > \bar{R}$, $\epsilon < \bar{\epsilon}$, $|\tau_ p| \leq \bar{\tau}$ and $\max_{1 \leq i \leq n} \phi_i(p,q) < \bar{\phi}$ then $d(\hat{p},\hat{q}) > \bar{\rho}$.
\end{lem}

The condition that $\tau_p$ is bounded away from $\pm 1$ is the crucial feature distinguishing this lemma, and its proof, from the similar second small scale separation lemma which follows. This condition ensures that $0$ and $q$, as in the statement, are not too close to the `poles' of $B_r(p)$ where the first order euclidean behaviour (corresponding to $\| z_p \|$) becomes negligible. This means that to prove this lemma we are able to just use the lower order terms to control the size of $z_q$, and hence ensure $\tau_q$ is large enough for $\hat{p}$ and $\hat{q}$ to be separated by an appropriate distance $\bar{\rho}$.

\begin{proof}
Using the isometries of $d$, see (\ref{isometries}), we may assume that $\tau_p \geq 0$  and $\phi(p) = 0$ without loss of generality. Setting $\phi = \phi(q)$ we therefore have 
	$$\max_{1 \leq i \leq n} |\phi_i| = \max_{1 \leq i \leq n} \phi_i(p,q) < \bar{\phi},$$ 
and of course $z_p = \re{z_p}$.

Our assumptions mean that
\begin{align*}
	p =\begin{pmatrix} (r+s) z_p \\ (r+s)^2 \tau_p \end{pmatrix} 
	&& \text{and} &&  
	q = \begin{pmatrix} (\tilde{r} + \tilde{s}) z_q  \\ (\tilde{r} + \tilde{s})^2 \tau_q \end{pmatrix}
\end{align*}
for some $s,\tilde{s}$ with $|s| \leq t$ and $|\tilde{s}| \leq \tilde{t}$. Let $a = \frac{r+s}{r+s'}$ and $b =  \frac{\tilde{r} + \tilde{s}}{r+s'}$. As $d(p,q) = r+s'$, some $|s'| \leq t$, using equation (\ref{sphere}) we know that
\begin{align*}
	1 &= 	\left\| a z_p - b z_q \right\|^2 
			+ \left( a^2 \tau_p - b^2 \tau_q - \frac{1}{2} a b \, \im{\langle z_p, z_q \rangle} \right)^2 \\
	&= 		a^2 \left\|  z_p \right\|^2 + b^2\left\|  z_q \right\|^2 - 2 a b \, \re{\langle z_p, z_q \rangle}    \\
	&		\phantom{bunchofstuf} + a^4 \tau_p^2 + b^4 \tau_q^2  + \frac{1}{4} \left(a b \, \im{\langle z_p, z_q 	\rangle} \right)^2 - \left( a^2 \tau_p - b^2 \tau_q \right)a b \, \im{\langle z_p, z_q \rangle} \tag{$\dagger$}
\end{align*}
where we have used the linearity properties of the inner product.
Observe that 
\begin{align*}
	a-1 = \frac{r+s}{r+s'} - 1 = \frac{\tilde{r}}{r} \frac{s/\tilde{r}-s'/\tilde{r}}{1+s'/r}  & &  \text{ and } & & b - \frac{\tilde{r}}{r} = \frac{\tilde{r} + \tilde{s}}{r+s'} - \frac{\tilde{r}}{r} = \frac{\tilde{r}}{r}\frac{\tilde{s}/\tilde{r} - s'/r }{1+s'/r}
\end{align*}
and since $r \geq \tilde{r} \geq t \tilde{t} R$ the rightmost fraction in each of these equalities is $O(R^{-1})$, independent of all other variables, as $R \to \infty$. By recalling that $\| z_p \|^2 + \tau_p^2 = 1$, and similarly with $q$, we can use this observation to reduce ($\dagger$) to
\begin{align*}
1 &= \left\|  z_p \right\|^2 - 2 \frac{\tilde{r}}{r} \re{\langle z_p, z_q \rangle} + \tau_p^2 - \tau_p \frac{\tilde{r}}{r} \im{\langle z_p, z_q \rangle} + O\left(\frac{\tilde{r}^2}{r^2}\right) + \frac{\tilde{r}}{r} E
\end{align*}
where $E$ is also an $O(R^{-1})$ error term. We can now subtract $\| z_p \|^2 + \tau_p^2 = 1$ and divide by a factor of $\frac{\tilde{r}}{r}$ to see that
\begin{align*}
2 \, \re{\langle z_p, z_q \rangle} + \tau_p \, \im{\langle z_p, z_q \rangle} = E + O\left( \frac{\tilde{r}}{r} \right).
\end{align*}
Note that $ \re{\langle z_p, z_q \rangle} = \sum_{j=1}^n \rho_j(p) \rho_j(q) \cos{\phi_j}$ and $\im{\langle z_p, z_q \rangle} = \sum_{j=1}^n \rho_j(p) \rho_j(q) \sin{\phi_j}$, and so if we take $\bar{\phi}$ small enough to ensure for each $j$ we have
$$\cos{\phi_j} + \tau_p \sin{\phi_j} \geq \cos{\phi_j} - |\sin{\phi_j}| \geq 0$$
then $0 \leq \re{\langle z_p, z_q \rangle} \leq 2 \, \re{\langle z_p, z_q \rangle} + \tau_p \, \im{\langle z_p, z_q \rangle}$ hence $|\re{\langle z_p, z_q \rangle}| \leq |E| + O\left( \frac{\tilde{r}}{r} \right)$.

Now suppose that $d(\hat{p},\hat{q}) \leq 1 - \bar{\tau}^2$ then $\| z_p - z_q \|^2 \leq 1 - \bar{\tau}^2$ and
\begin{align*}
\tau_q^2 = 1 - \| z_q \|^2 
&= 		1 + \| z_p \|^2 - \| z_p - z_q \|^2 - 2 \, \re{\langle z_p, z_q \rangle}  \\
&\geq 	1 + (1 - \bar{\tau}^2) - (1 - \bar{\tau}^2) - 2 \, \re{\langle z_p, z_q \rangle} \\
&> \frac{1 + \bar{\tau}^2}{2} > \bar{\tau}^2 \geq \tau_p^2
\end{align*}
where we have ensured that $\bar{\epsilon}$ and $\bar{R}^{-1}$ small enough for $4 |\re{\langle z_p, z_q \rangle} | < 1 - \bar{\tau}^2$.
It follows that
$$d(\hat{p},\hat{q}) > \frac{1}{2} d\left(\lbrace h \in \partial B_1(0) : \tau_h^2 \leq \bar{\tau}^2 \rbrace ,\left\lbrace h \in \partial B_1(0) : \tau_h^2 \geq \frac{1 + \bar{\tau}^2}{2} \right\rbrace \right) > 0$$
and so we also take $\bar{\rho} > 0$ as the minimum of $1 - \bar{\tau}^2$ and this value to complete the proof. 
\end{proof}

\begin{lem}[Small scale separation 2] \label{largetau}
There exists $\bar{\tau} \in (\frac{1}{2},1)$ and $\bar{\rho},\bar{R},\bar{\phi},\bar{\epsilon} > 0$ for which the following holds. Let $p,q \in \mathbb{H}^n \setminus \lbrace 0 \rbrace$ with $q \in \tb{t}{r}{p}$ and $0 \in \tb{t}{r}{p} \cap \tb{\tilde{t}}{\tilde{r}}{q}$ where $t, \tilde{t} \geq 1$ and suppose $r \geq \tilde{r} \geq T R$ for some $R > 1$ and $T \geq t \tilde{t}$. Let $I(p) = \left\lbrace i : \rho_i(p) < \frac{10T}{r}  \right\rbrace$, $\epsilon \in (0,1)$ and assume that $\tilde{r} \leq \epsilon r$. If $R > \bar{R}$, $\epsilon < \bar{\epsilon}$, $|\tau_ p| \geq \bar{\tau}$ and $\max_{1 \leq i \leq n} \phi_i(p,q) < \bar{\phi}$, then either there exists $i \not \in I_p$ such that $\rho_i(q) < \frac{10T}{\tilde{r}}$ or $d(\hat{p},\hat{q}) > \bar{\rho}$.
\end{lem}

In this lemma we aim for the same conclusion as in the first small scale separation lemma but find an exceptional case. This we deal with later by using a slightly more sophisticated bounding argument.

As remarked above, in the setting of this lemma the first order argument used to prove Lemma \ref{smalltau} is not available to us; the argument stalls if $z_p$ can be made arbitrarily small. Instead we must make delicate use of the precise shape of $B_r(p)$ near the poles. This results in a somewhat more technical proof where special care must be paid to the thickenings, which in this case are large enough to easily throw off the estimates.

\begin{proof}
As before we may assume that $\tau_p \geq 0$  and $\phi(p) = 0$, so $z_p = \re{z_p}$, without loss of generality. Again we set $\phi = \phi(q)$ so that $\max_{1 \leq i \leq n} |\phi_i| = \max_{1 \leq i \leq n} \phi_i(p,q) < \bar{\phi}$.

To keep track of a large quantity of error terms, we will slightly abuse the big $O$ and little $o$ notations. Throughout we shall write $O(x)$ for any function $f : \mathbb{R} \to \mathbb{R}$ (possibly depending on our variables) for which, by first taking $\bar{\tau}$ sufficiently close to $1$, then $\bar{\epsilon}$ sufficiently small and $\bar{R}$ sufficiently large we can ensure $|f(x)| \leq K |x|$ for some $K > 0$ independent of all other variables. Similarly we will write $o(x)$ for any function $f(x)$ for which given any $\delta > 0$, with the same control over $\bar{\epsilon}, \bar{R}$ and $\bar{\tau}$, we can ensure $|f(x)| \leq \delta |x|$.

Our approach is to attempt to bound $\tau_q$ above by some constant $C < 1$. When successful, it will then suffice to take 
$$\bar{\rho} < \frac{1}{2} d\left( (0,0,1), \lbrace h \in \partial B_1 : \tau_h \leq C \rbrace \right)$$ 
since $\bar{\tau}$ can be increased to ensure $\hat{p}$ is arbitrarily close to $(0,0,1)$. We will encounter an exceptional case to account for the `either'. Firstly, if $\|z_q \| \geq \frac{1}{2}$ then $\tau_q^2 \leq \frac{3}{4}$ so we may assume that $\|z_q \| \leq \frac{1}{2}$.

\begin{flushleft}
	\textbf{Step 1: Perturb $p$ and $q$ to suppress the thickenings}
\end{flushleft}

We are now going to introduce some new points which incorporate the errors due to the thickenings; this enables us to keep the errors under sufficient control to be dealt with later. Let $\eta \in B_t(0)$ such that $d(\eta^{-1}, p) = r$ and $q' \in B_t(q)$ such that $d(q',p) = r$. Let $P = p \, \eta$ and $Q = q' \eta$. For notational simplicity we let $P = (z,\tau)$ and $Q = (w,\sigma)$. We can write these variables more explicitly using the coordinates of $p, q$ and $\eta$: for some $s,s_\eta$ with $|s|,|s_\eta| \leq t$ we have
\begin{align}
z = (r+s)z_p + s_\eta z_\eta & & \text{and} & & \tau = (r+s)^2\tau_p + s_\eta^2 \tau_\eta + \frac{1}{2} \im{\langle (r+s)z_p ,s_\eta z_\eta  \rangle }, \label{z&taucoord}
\end{align}
and (using $d(q,q') \leq t$) there are $\zeta \in \mathbb{C}^n$ and $\zeta_\tau \in \mathbb{R}$ such that $\| \zeta \|, |\zeta_\tau| \leq t$ for which
\begin{align}
w = (\tilde{r} + \tilde{s}) z_q + \zeta + s_\eta z_\eta \label{wcoord}
\end{align}
and
\begin{align*}
\sigma &=  \left[ (\tilde{r} + \tilde{s})^2 \tau_q + \frac{1}{2} \im{\langle(\tilde{r} + \tilde{s}) z_q + \zeta,(\tilde{r} + \tilde{s}) z_q   \rangle} + \zeta_\tau \right] + s_\eta^2 \tau_\eta + \frac{1}{2} \im{\langle (\tilde{r} + \tilde{s}) z_q + \zeta ,s_\eta z_\eta \rangle}.
\end{align*}
This final expression is somewhat complicated, but by considering the dominant $\tilde{r}^2$ term and noting $t,\tilde{t} \leq R^{-1} \tilde{r}$ it becomes clear that
$$\sigma = \tau_q \tilde{r}^2 + o(\tilde{r}^2) = O(\tilde{r}^2).$$
Therefore, to bound $\tau_q$ above by some $C < 1$ it will suffice to do so for $\frac{\sigma}{\tilde{r}^2}$.

To do this we will use the fact that, by the right invariance of the metric, $d(P,Q) = r$ and $d(0,P) = r$.  The first of these properties ensures that 
\begin{align*}
\frac{\|z - w \|^2}{r^2} + \frac{( \tau - \sigma - \frac{1}{2}\im{ \langle z ,w \rangle })^2}{r^4} = 1
\end{align*}
and hence that $\sigma$ is given by one of the roots
\begin{align*}
\tau - \frac{1}{2}\im{ \langle z,w \rangle } \pm r^2 \sqrt{1 - r^{-2}\|z - w \|^2}.
\end{align*}

\begin{flushleft}
	\textbf{Step 2: Apply Taylor's theorem to the square root}
\end{flushleft}

The fact that $d(0,P) = r$ means
\begin{align*}
\frac{\|z \|^2}{r^2} + \frac{\tau^2}{r^4} = 1
\end{align*}
and so as long as $\tau > 0$, which we will see just below, we have
\begin{align*}
r^2 \sqrt{1 - r^{-2}\|z - w \|^2} = \tau \sqrt{1 - \frac{r^2}{\tau^2}\left( \| w \|^2 - 2\re{\langle z,w \rangle} \right)}.
\end{align*}
Using the coordinates expressions in (\ref{z&taucoord})
\begin{align*}
\frac{\tau}{r^2} 
= \left(1+\frac{s}{r}\right)^2 \tau_p+ \frac{s_\eta^2 }{r^2}\tau_\eta +  \frac{1}{2} \left(1+\frac{s}{r} \right)  \frac{s_\eta}{r} \im{\langle  z_p,z_\eta  \rangle }
= \tau_p + o(1).
\end{align*}
In particular are able to ensure that $0 < \frac{1}{2}\bar{\tau} \leq r^{-2}\tau \leq \frac{3}{2}$. Additionally, we have that
$$\left\| \frac{w}{r} \right\| \leq  \frac{\tilde{r} + |\tilde{s}|}{r}    + \frac{t}{r} +  \frac{|s_\eta|}{r}   \leq \epsilon + \frac{3}{R}$$
and so $\left\| \frac{w}{r} \right\| = o(1)$. As 
\begin{align*}
r^{-2} \left| \| w \|^2 - 2\re{\langle z,w \rangle} \right| \leq   \left\| \frac{w}{r} \right\| \left(1 + 2 \left\| \frac{w}{r} \right\| \right)
\end{align*}
we also have $r^{-2} (\| w \|^2 - 2\re{\langle z,w \rangle}) = o(1)$. Combining these observations shows that
\begin{align*}
E = \frac{r^2}{\tau^2}\left( \| w \|^2 - 2\re{\langle z,w \rangle} \right) = \frac{1}{(r^{-2}\tau)^2}r^{-2}\left( \| w \|^2 - 2\re{\langle z,w \rangle} \right) = o(1).
\end{align*}
With $\bar{\tau}$, $\bar{\epsilon}$ and $\bar{R}$ sufficiently well chosen we can therefore apply Taylor's theorem to see that
\begin{align*}
\tau \sqrt{1 - \frac{r^2}{\tau^2}\left( \| w \|^2 - 2\re{\langle z,w \rangle} \right)}
= \tau\sqrt{1 - E} = \tau \left(1 - \frac{E}{2} + O\left( E^2 \right) \right)
\end{align*}
and hence that $\sigma$ is one of
$$\tau - \frac{1}{2}\im{ \langle z,w \rangle } \pm \left(\tau - \frac{\tau E}{2} + O\left( \tau E^2 \right) \right).$$

In principle, as $\tau = r^2 \tau_p + o(r^2)$ and $\tau_p \geq \bar{\tau} > \frac{1}{2}$, $\tau$ can be very large. This will cause problems bounding $\sigma$ if it is ever given by the positive root. However, we have already seen that $\sigma = O(\tilde{r}^2)$, $\tau = O(r^2)$ and $E = o(1)$ and so if $\sigma$ were given by the positive root then $ \tau = o(r^2)$, contradicting $\tau_p > \frac{1}{2}$. Therefore
$$\sigma = - \frac{1}{2}\im{ \langle z,w \rangle }  + \frac{\tau E}{2} + O\left( \tau E^2 \right)$$
and we just need to find appropriate bounds for the remaining terms.

\begin{flushleft}
	\textbf{Step 3: Bounding $\tau E^2$}
\end{flushleft}

Observe that
\begin{align*}
\tau E^2 &= \frac{r^4}{\tau^3}\left( \| w \|^2 - 2\re{\langle z,w \rangle} \right)^2 \\
		&\leq  \frac{\tilde{r}^2}{(r^{-2}\tau)^3} \left( \left\| \frac{w}{r} \right\| \left\| \frac{w}{\tilde{r}} \right\| + 2\left|\re{\left\langle \frac{z}{r},\frac{w}{\tilde{r}} \right\rangle}\right| \right)^2 \\
		&\leq  2^6 \left\| \frac{w}{\tilde{r}} \right\|^2  \left( \left\| \frac{w}{r} \right\|   +2\left(1+ \frac{|s|}{r} \right)\|z_p\|+ \frac{2|s_\eta|}{r}\right)^2  \tilde{r}^2 
\end{align*}
where we have used $r^{-2}\tau \geq \frac{1}{2}\bar{\tau} \geq \frac{1}{4}$ and the coordinate expression for $z$. Noting that
$$\left\| \frac{w}{\tilde{r}} \right\| \leq   1 + \frac{|\tilde{s}|}{\tilde{r}} + \frac{t}{\tilde{r}} +  \frac{|s_\eta|}{\tilde{r}}  \leq 2$$
for $\epsilon$ and $R^{-1}$ sufficiently small, and that $\|z_p\| = \sqrt{1-\tau_p^2}$, we see that
\begin{align*}
\tau E^2 \leq  2^8 \left( \epsilon + \frac{5}{R} + 2\left(1+ R^{-1} \right) \sqrt{1-\bar{\tau}^2} \right)^2 \tilde{r}^2.
\end{align*}
This means $\tau E^2 = o(\tilde{r}^2)$, which will be sufficient.

\begin{flushleft}
	\textbf{Step 4: Bounding the explicit terms in the non-exceptional case}
\end{flushleft}

This is the most technical step in the proof, but is not fundametally difficult.
\begin{align*}
-\frac{1}{2}\im{ \langle z,w \rangle } + \frac{\tau E}{2} 
    &= -\frac{1}{2}\im{ \langle z,w \rangle } + \frac{r^2}{2\tau}\left( \| w \|^2 - 2\re{\langle z,w \rangle} \right) \\
    &=  \frac{1}{2}\left( -\im{ \langle z,w \rangle } - 2 \frac{r^2}{\tau} \re{\langle z,w \rangle} \right) + \frac{r^2}{2\tau} \| w \|^2.
\end{align*}
We aim to show the term inside the bracket is non-positive, modulo a small error term. We have
\begin{align*}
\langle z,w \rangle 
&= 	\langle (r+s)z_p,(\tilde{r} + \tilde{s}) z_q + \zeta + s_\eta z_\eta \rangle + 		o(\tilde{r}^2) \\
&= 	(r+s)\tilde{r} \left\langle z_p,\left(1 + \frac{\tilde{s}}{\tilde{r}}\right) z_q + \zeta + \frac{s_\eta}{\tilde{r}} z_\eta \right\rangle + o(\tilde{r}^2).
\end{align*}
Next
\begin{align*}
&\im{\left\langle z_p,\left(1 + \frac{\tilde{s}}{\tilde{r}}\right) z_q + \zeta + \frac{s_\eta}{\tilde{r}} z_\eta \right\rangle} \\
= &\sum_{j=1}^n \rho_j(p) \left( \left( 1 + \frac{\tilde{s}}{\tilde{r}} \right) \rho_i(q)\sin{\phi_j}  + \im{\frac{\zeta_i + s_\eta (z_\eta)_i}{\tilde{r}}} \right)\\
= &\sum_{j \not\in I(p)} \rho_j(p) \left( \left( 1 + \frac{\tilde{s}}{\tilde{r}} \right) \rho_i(q)\sin{\phi_j}  + \im{\frac{\zeta_i + s_\eta (z_\eta)_i}{\tilde{r}}} \right) + o\left( \frac{\tilde{r}}{r} \right)
\end{align*}
and
\begin{align*}
&\re{\left\langle z_p,\left(1 + \frac{\tilde{s}}{\tilde{r}}\right) z_q + \zeta + \frac{s_\eta}{\tilde{r}} z_\eta \right\rangle} \\
= &\sum_{j=1}^n \rho_j(p) \left( \left( 1 + \frac{\tilde{s}}{\tilde{r}} \right) \rho_i(q)\cos{\phi_j}  + \re{\frac{\zeta_i + s_\eta (z_\eta)_i}{\tilde{r}}} \right)\\
= &\sum_{j \not\in I(p)} \rho_j(p) \left( \left( 1 + \frac{\tilde{s}}{\tilde{r}} \right) \rho_i(q)\cos{\phi_j}  + \re{\frac{\zeta_i + s_\eta (z_\eta)_i}{\tilde{r}}} \right) + o\left( \frac{\tilde{r}}{r} \right).
\end{align*}
From earlier assumptions $\| \zeta + s_\eta z_\eta \| \leq 2t$, and we can ensure that
\begin{align*}
-\sin{\phi_j} - 2 \frac{r^2}{\tau} \cos{\phi_j} \leq -1
\end{align*}
through further increasing $\bar{\tau}$ and then decreasing $\bar{\epsilon}$, $\bar{R}^{-1}$ and $\bar{\phi}$. So it will be enough for the magnitude of each $\rho_i(q)$ to be large relative to $2t$. In the non-exceptional case we may assume that for all $j \not\in I(p)$ we have $\rho_j(q) \geq 10T\tilde{r}^{-1}$ so that
\begin{align*}
\left( 1 + \frac{\tilde{s}}{\tilde{r}} \right) \rho_j(q) \geq \rho_j(q) \geq 10t\tilde{r}^{-1} \geq 2t\tilde{r}^{-1} \left( 1 + \frac{2r^2}{\tau} \right) 
\end{align*}
since we ensured $r^{-2}\tau > \frac{1}{2}$. It follows that 
\begin{align*}
-\sum_{j \not\in I_p} \rho_j(p) &\left( \left( 1 + \frac{\tilde{s}}{\tilde{r}} \right) \rho_i(q)\sin{\phi_j}  + \im{\frac{\zeta_i + s_\eta (z_\eta)_i}{\tilde{r}}} \right) \\
&- \frac{2r^2}{\tau} \sum_{j \not\in I_p} \rho_j(p) \left( \left( 1 + \frac{\tilde{s}}{\tilde{r}} \right) \rho_i(q)\cos{\phi_j}  + \re{\frac{\zeta_i + s_\eta (z_\eta)_i}{\tilde{r}}} \right)\leq 0.
\end{align*}
This means that the explicit terms are the sum of something non-positive and an error term with order 
\begin{align*}
	 (r+s)\tilde{r} \, o\left( \frac{\tilde{r}}{r} \right) + o(\tilde{r}^2) = o(\tilde{r}^2).
\end{align*}

\begin{flushleft}
	\textbf{Step 5: Bound $\sigma$ in the non-exceptional case and complete the proof}
\end{flushleft}

By combining the last two steps we see that
\begin{align*}
\sigma 
&=  	\frac{1}{2}\left( -\im{ \langle z,w \rangle } - 2 \frac{r^2}{\tau} \re{\langle z,w \rangle} \right) + \frac{r^2}{2\tau} \| w \|^2 + O\left(\tau E^2\right) \\
&\leq 	\frac{r^2}{2\tau} \| w \|^2 + o(\tilde{r}^2) \\
&\leq \frac{\tilde{r}^2}{4\bar{\tau}} + o(\tilde{r}^2)
\leq \frac{\tilde{r}^2}{2} + o(\tilde{r}^2)
\end{align*}
which allows us to bound $\sigma$ in the required fashion unless we have some $j \not\in I_p$ for which $\rho_j(q) < \frac{10T}{\tilde{r}}$, which is the other option allowed by the statement.
\end{proof}

\subsection{Finite intersection dimension}

We can now fit these pieces together to show that property (iv) holds.

\begin{theorem}
$(\mathbb{H}^n,d)$ has finite intersection dimension.
\end{theorem}

\begin{proof}
We need to show that there exists $R > 1$ and $\kappa \in \mathbb{N}$ such that if we are given
\begin{enumerate}
\item $t(1),...,t(\kappa) \geq 1$,
\item $r(1),...,r(\kappa)$ such that each $r(i) \geq t(1)...t(i)R$,
\item points $p_1,...,p_\kappa \in \mathbb{H}^n$ such that $p_i \in \bigcap_{j < i} \partial_{t(j)} B_{r(j)}(p_j)$ for $j < i$,
\end{enumerate}
then $\bigcap_{i=1}^\kappa \partial_{t(i)} B_{r(i)}(p_i) = \emptyset$. Let us assume, by using invariance, that $0 \in \bigcap_{i=1}^\kappa \partial_{t(i)} B_{r(i)}(p_i)$ and show that $k$ must be bounded for $R$ sufficiently large.

The logical structure of the proof is to first apply a number of reductions of the form: we have a sequence of length $\kappa$ with a collection of properties $P$, we show that given $\kappa'$ there is $M(\kappa') \in \mathbb{N}$ such that if $\kappa \geq M$ then there is a subsequence of length $\kappa'$ with a property $Q$ in addition to those properties in $P$. It is then sufficient to show $\kappa'$ is bounded, because if so it follows that $\kappa < M(\kappa' + 1)$ where $\kappa'$ is maximal with the properties in $P$ and $Q$ holding. We can then relabel and assume our sequence had property $Q$ in the first place. We finish off by using all the gathered properties to show $\kappa$ is bounded.

\begin{flushleft}
\textbf{Reduction 1:}
\end{flushleft} 

First we show that we can assume the $r(i)$ are decreasing, essentially as in \cite{Hoch1}. Let $\kappa' \leq \kappa$ and assume that $r(i) \geq r(1)$ for all $2 \leq i \leq \kappa'$. By property 3 all these $p_i$ lie inside $\partial_{t(1)} B_{r(1)}(p_1) \subset B_{2r(1)}(p_1)$, this containment is due to property 2. Property 3 also ensures that for pair $i,j$ with $j > i$ there is a point $b \in\partial B_{r(i)}(p_i)$ with $d(b,p_j) \leq t(i)$, and hence by property 2
\begin{align*}
d(p_i,p_j) \geq |d(p_i,b) - d(p_j,b)| \geq r(i) - t(i) \geq r(1)(1 - R^{-1})
\end{align*}
so for all $1 \leq i,j \leq k'$ with $i \neq j$, 
\begin{align*}
d\left(p_1^{-1}\delta_{1/(2r(1))}p_i , \, p_1^{-1}\delta_{1/(2r(1))}p_j\right) \geq \frac{1 - R^{-1}}{2} > 0.
\end{align*}
Since each $p_1^{-1}\delta_{1/(2r(1))}p_i \in B_1(0)$ by Lemma \ref{sep-pt} part (i) $\kappa' \leq N_R$, where $N_R = N(\frac{1 - R^{-1}}{2})$. Note that $N_R$ decreases as $R$ increases.

Clearly, this argument could be repeated with any chain of $\kappa'$ points satisfying the analogous conditions. Therefore if for some $\kappa'' \in \mathbb{N}$ we have $\kappa \geq \kappa''(N_R+1)$ then there must be $i_1 = 1 < i_2 \leq ... < i_{\kappa''} \leq \kappa$ with $r(i_1) \geq r(i_2) \geq ... \geq r(i_{\kappa''})$. This means it suffices for us to prove the claim with the $r(i)$ assumed to be decreasing.

\begin{flushleft}
\textbf{Reduction 2:}
\end{flushleft}

Next we use Lemma \ref{sep-pt} and the large scale separation lemma, \ref{LSS}, to ensure that we can assume that $\hat{p}_1,...,\hat{p}_\kappa$ are all within a distance 
$$\rho(\epsilon) = \frac{1}{2}\left(1 - \sqrt{1-\frac{\epsilon^2}{4}} \right) > 0$$
of one another, here $\epsilon \in (0,1)$, and that for all $j > i$ we have $r(j) \leq \epsilon r(i)$. Note that $\rho(\epsilon)$ decreases as $\epsilon$ decreases.

Let $\kappa' \leq \kappa$, again. By Lemma \ref{sep-pt} part (ii) if $\kappa \geq \kappa' N(\rho(\epsilon))$ then we have a subcollection $I \subset \lbrace 1,...,\kappa \rbrace$ of size at least $\kappa'$ with $d(\hat{p}_i,\hat{p}_j) < \rho(\epsilon)$ for all $i,j \in I$. By taking $R > \bar{R}(\epsilon)$ from Lemma \ref{LSS}, which is assumed to hold from here onwards, the lemma shows that for each pair $i,j \in I$ with $j < i$ we have $r(j) \leq \epsilon r(i)$. $I$ therefore gives the desired subsequence.

\begin{flushleft}
\textbf{Reduction 3:}
\end{flushleft}

Before we do our final reduction, first take $\bar{\tau}$ as given by Lemma \ref{largetau}, and we take this as the input for $\bar{\tau}$ in Lemma \ref{smalltau}. We can then decrease $\epsilon$ so that $\epsilon$ and $\rho(\epsilon)$ are small enough to apply Lemmas \ref{smalltau} and \ref{largetau} with $\bar{\epsilon} = 2\epsilon$ and $\bar{\rho} = \rho(\epsilon)$. Similarly we take $R$ large enough for both lemmas to hold.

It should be clear from an application of the pigeonhole principle that given $\kappa'$ by increasing $\kappa$ we can ensure that there is a subcollection $I \subset \lbrace 1,...,\kappa \rbrace$ of size $\kappa'$ such that for all $i,j \in I$ we have $\max_{1 \leq l \leq n} \phi_l(p_i,p_j) < \bar{\phi}$, where $\bar{\phi}$ is small enough for both lemmas to hold. We can therefore assume the whole sequence also has this property.

\begin{flushleft}
\textbf{$\kappa$ is bounded:}
\end{flushleft}

With all this in hand, we can apply Lemmas \ref{smalltau} and \ref{largetau} to the sequence at will. Let $T = t_1...t_\kappa$ and for each $1 \leq i \leq \kappa$ set $I(p_i) = \lbrace m : \rho_m(p_i) < \frac{10T}{r_i} \rbrace \subseteq \lbrace 1,...,n \rbrace$ as in Lemma \ref{largetau}. By assumption for all $i \neq j$ we have $d(\hat{p}_i,\hat{p}_j) \leq \bar{\rho}$ and so by Lemma \ref{smalltau} we must have $|\tau_{p_i}| > \bar{\tau}$ for all $i \leq \kappa-1$. By applying this fact along with the same assumption Lemma \ref{largetau} ensures that for each pair $i < j \leq \kappa$ there is some number in $I(p_j)$ which is not in $I(p_i)$. In particular, each of the sets $I(p_1),...,I(p_\kappa) \subseteq \lbrace 1,...,n \rbrace$ are pairwise distinct, from which it follows that $\kappa \leq 2^n$.
\end{proof}

Having completed this proof all that remains is property (i), well-separability.

\subsection{Well-separability}

We begin with a preliminary lemma.

\begin{lem} \label{closeball}
Let $p, p' \in \mathbb{H}^n$ and $r > 0$. Then there exists $R > 0$, independent of $p$, $p'$ and $r$, such that if $\rho = d(p,p') > 2Rr$ then there is a point $q$ with $d(p',q) \leq 2r$ for which $B_{r}(q) \subseteq B_\rho(p)$.
\end{lem}

\begin{proof}
First of all, using the dilation and isometries of $d$ we may assume that $r = 1/2$ and $p' = 0$. Moreover we assume that $\tau_p \geq 0$ and all coefficients $z_p$ are non-negative reals.

By right invariance the points in $B_{1/2}(q)$ take the form $(w + q_z, \sigma + q_\tau + \frac{1}{2}\im{\langle w, q_z \rangle})$ where $\|w\|^2 + 4\sigma^2 \leq \frac{1}{4}$. Therefore, by (\ref{ball}), it suffices to show that we can choose $R$ large enough such that given $(z_p,\tau_p)$ there is $q = (q_z,q_\tau)$ with $\| q_z \|^2 + q_\tau^2 \leq 1$ such that
\begin{align*}
\frac{\left\| w + q_z - \rho z_p  \right\|^2}{\rho^2} + \frac{\left( \sigma + q_\tau + \frac{1}{2}\im{\langle w, q_z \rangle} - \rho^2 \tau_p - \frac{1}{2} \im{\langle w + q_z, \rho z_p \rangle} \right)^2}{\rho^4} \leq 1
\end{align*}
or equivalently (as $d(0,p) = \rho$) that
\begin{align*}
0 \geq \rho^3 &\left( - 2\re{\langle w + q_z ,z_p  \rangle} + \tau_p \im{\langle w + q_z,z_p  \rangle} \right) \\ 
&+ \rho^2 \left( \| w + q_z \|^2 - 2 \tau_p \left( \sigma + q_\tau + \frac{1}{2}\im{\langle w, q_z \rangle} \right) + \frac{1}{4} (\im{\langle w + q_z,  z_p \rangle})^2 \right) \\
&\phantom{+ \rho^2} -  \frac{\rho}{2} \left( \sigma + q_\tau + \frac{1}{2}\im{\langle w, q_z \rangle} \right)  \im{\langle w + q_z, \rho z_p \rangle} + \left( \sigma + q_\tau + \frac{1}{2}\im{\langle w, q_z \rangle} \right)^2.
\end{align*}
Notice that the coefficients of all powers of $\rho$ have bounds independent of all variables. Let $C > 0$ be strictly greater than the independent bound for the coefficient of $\rho^2$ and ensure that $R > C$. Consider the case when $\|z_p \| \geq \frac{2C}{\rho} > 0$, let us take $q_z = \lambda z_p$ where $\lambda > 0$ is chosen so that $\| q_z \| = 1$, and hence $q_\tau = 0$. Then the coefficient of $\rho^3$ above satisfies
\begin{align*}
- 2\re{\langle w + q_z ,z_p  \rangle} + \tau_p \im{\langle w + q_z,z_p  \rangle} 
&= -2 \langle q_z,z_p \rangle - \langle z_p, 2 \re{w} + \tau_p \im{w} \rangle \\
&\leq \| z_p \| (-2 + \frac{3}{2}) \leq - \frac{C}{\rho}.
\end{align*}
It follows that the polynomial above is bounded above by a quadratic in $\rho$ whose coefficients are independent of all variables, and the leading coefficient of which is negative. Hence we may take $R$ large enough, with the required independence, to ensure that the inequality holds for some appropriate $q$ regardless of the choice of $p$.

In the case where $\|z_p \| \leq \frac{2C}{\rho}$ take $q_z = 0$ and $q_\tau = 1$. Then we have the bounds
\begin{align*}
- 2\re{\langle w + q_z ,z_p  \rangle} + \tau_p \im{\langle w + q_z,z_p  \rangle} \leq  \frac{3C}{\rho}
\end{align*}
and
\begin{align*}
\| w + q_z \|^2 - 2 \tau_p \left( \sigma + q_\tau + \frac{1}{2}\im{\langle w, q_z \rangle} \right) + \frac{1}{4} (\im{\langle w + q_z,  z_p \rangle})^2 \leq \frac{1}{4} - \frac{3}{2} \tau_p + \frac{C^2}{4\rho^2}.
\end{align*}
In particular, we can show that the above polynomial is bounded above by a quadratic whose coefficients are independent of all variables and whose leading coefficient is less than
\begin{align*}
\frac{3C}{\rho} + \frac{1}{4} - \frac{3}{2} \sqrt{1 - \frac{C^2}{\rho^2}} + \frac{C^2}{4\rho^2} \leq -1
\end{align*}
where $R$ has been taken sufficiently large relative to $C$. So, as above we may increase $R$ to ensure the required inequality holds.
\end{proof}

We call a sequence of balls in a metric space \textit{incremental} if the radii are non-increasing and the centre of each ball is not an element any ball earlier in the sequence. In particular, each centre is only in one ball in the sequence.

\begin{prop}
$(\mathbb{H}^n,d)$ is well-separable.
\end{prop}

\begin{proof}
We mildly adapt a standard technique, see for example \cite{Hoch1} or \cite{Guz1}. For the purposes of this proof we use $\nu$ to denote the right invariant Haar measure on $\mathbb{H}^n$.

Let $C$ be the constant of the Besicovitch covering property and $D$ be the constant for the \textit{metric} doubling property of $d$. Furthermore, take $m \in \mathbb{N}$ large enough so that $2^m > R$, with $R$ as in Lemma \ref{closeball}. In particular, $m$ depends only on the metric $d$. Let $\chi = C D^{m+2} + 1$.

Let $E$ be a finite subset of $\mathbb{H}^n$ and $\mathcal{U}$ be a carpet covering $E$. By applying the BCP via, for example, proposition 2.1 of \cite{Hoch1} we can find an incremental sequence $U_1,...,U_n$ of elements of $\mathcal{U}$ covering $E$. We assign colours $1,2,...,\chi$ to the $U_i$ as follows. Colour $U_1$ as you like, assume we have coloured the $U_i$ for $i \leq k$ and consider $U_{k+1}$. Take $r$ to be the radius of $U_k$ and $h$ to be the centre of $U_{k+1}$, by assumption $U_{k+1} \subseteq B_r(h)$ and each $U_i$ with $i \leq k$ has radius at least $r$. 

Let $\mathcal{W}$ be the collection of balls $U_1,...,U_k$ which are within distance $r$ of $U_{k+1}$, and hence of $B_r(h)$, and let $N = |\mathcal{W}|$. Each $U \in \mathcal{W}$ intersects nontrivially with $B_{2r}(h)$, we may take $p'$ from Lemma \ref{closeball} to be a point in this intersection. We may assume that $p'$ is on the boundary of $U$ because the straight line from $h$ to $p'$ is contained by $B_{2r}(h)$ (the balls are euclidean convex), $p' \in U$ and $h \not\in U$ (by incrementality) so the intermediate value theorem implies there is a point on the boundary of $U$ inside $B_{2r}(h)$. The Lemma \ref{closeball} then ensures that either the radius of $U$ is at most $2^{m+1} r$ we can replace $U$ with a ball of radius $r$ centred in $B_{4r}(h)$, call this new collection of balls $\mathcal{W}'$. In particular, each ball in $\mathcal{W}'$, which is also of size $N$, has radius at least $r$ and is contained by the ball of radius $2^{m+2}r$ about $h$. Therefore by the Besicovitch and metric doubling properties
\begin{align*}
N\nu(B_r(0)) \leq C\nu(B_{2^{m+2}r}(h)) \leq C D^{m+2} \nu(B_r(h)) = C D^{m+2} \nu(B_r(0))
\end{align*}
and so $N \leq C D^{m+2}$. Since $N \leq \chi - 1$ we assign a colour $U_k$ which is different to all those within distance $r$ of $U_k$.

Once the colouring is complete, the collection $\mathcal{V}_j$ of those balls coloured $j$ is well separated precisely because of this property combined with the fact that the radii are decreasing.
\end{proof}

This proposition completes the proof of Theorem \ref{HeisET}.

\singlespacing
\bibliographystyle{alpha} 
\bibliography{Bibliography}

\end{document}